\documentclass{amsart}
\usepackage{url}
\usepackage{mathrsfs}
\usepackage{booktabs}

\usepackage{comment}
\usepackage{color}
\usepackage{graphicx}
\usepackage{caption}
\usepackage{subcaption}
\usepackage{multirow}
\usepackage{mathtools}
\usepackage{amssymb,tikz}
\usepackage{ctable}
\usepackage{amsmath}
\usepackage{circuitikz}

\newcommand{\re}{\mathbb{R}}

\newcommand{\N}{\mathbb{N}}

\def\af{\alpha}

\def\rank{\mbox{rank}}

\newcommand{\Sig}{\Sigma}

\newcommand{\reff}[1]{(\ref{#1})}

\newcommand{\mc}[1]{\mathcal{#1}}

\newcommand{\supp}[1]{\mbox{supp}(#1)}
\newcommand{\cone}[1]{\mathit{cone}(#1)}
\renewcommand{\vec}[1]{\mathit{vec}(#1)}
\newcommand{\qmod}[1]{\mbox{QM}[#1]}
\newcommand{\st}{\mathit{s.t.}}

\newcommand{\bdes}{\begin{description}}
	\newcommand{\edes}{\end{description}}
\newcommand{\bal}{\begin{align}}
	\newcommand{\eal}{\end{align}}
\newcommand{\bnum}{\begin{enumerate}}
	\newcommand{\enum}{\end{enumerate}}
\newcommand{\bit}{\begin{itemize}}
	\newcommand{\eit}{\end{itemize}}
\newcommand{\bea}{\begin{eqnarray}}
	\newcommand{\eea}{\end{eqnarray}}
\newcommand{\be}{\begin{equation}}
	\newcommand{\ee}{\end{equation}}
\newcommand{\baray}{\begin{array}}
	\newcommand{\earay}{\end{array}}
\newcommand{\bsry}{\begin{subarray}}
	\newcommand{\esry}{\end{subarray}}
\newcommand{\bca}{\begin{cases}}
	\newcommand{\eca}{\end{cases}}
\newcommand{\bcen}{\begin{center}}
	\newcommand{\ecen}{\end{center}}
\newcommand{\bbm}{\begin{bmatrix}}
	\newcommand{\ebm}{\end{bmatrix}}
\newcommand{\bmx}{\begin{matrix}}
	\newcommand{\emx}{\end{matrix}}
\newcommand{\bpm}{\begin{pmatrix}}
	\newcommand{\epm}{\end{pmatrix}}
\newcommand{\btab}{\begin{tabular}}
	\newcommand{\etab}{\end{tabular}}

\newtheorem{theorem}{Theorem}[section]

\newtheorem{cor}[theorem]{Corollary}

\theoremstyle{definition}
\newtheorem{example}[theorem]{Example}

\newtheorem{algorithm}[theorem]{Algorithm}

\numberwithin{equation}{section}

\begin{document}

\title[Distributionally Robust Optimization]
	{Distributionally Robust Optimization with
		Polynomial Robust Constraints}

	\author[Jiawang Nie]{Jiawang Nie}
	\address{Department of Mathematics, 
    University of California San Diego, 9500 Gilman Drive,
    La Jolla, 92093, CA, USA.}
	\email{njw@math.ucsd.edu}

	\author[Suhan Zhong]{Suhan~Zhong}
	\address{Suhan Zhong, Department of Mathematics,
		Texas A\&M University, College Station, TX, USA, 77843-3368.}
	\email{suzhong@tamu.edu}
	
	
	\date{}
	

	\begin{abstract}
		This paper studies distributionally robust optimization (DRO)
		with polynomial robust constraints.
		We give a Moment-SOS relaxation approach to solve the DRO.
		This reduces to solving linear conic optimization with semidefinite constraints.
		When the DRO problem is SOS-convex, we show that
		it is equivalent to the linear conic relaxation
		and it can be solved by the Moment-SOS algorithm.
		For nonconvex cases, we also give concrete conditions
		such that the DRO can be solved globally.
		Numerical experiments are given to show the efficiency of the method.
	\end{abstract}

	\maketitle

	\section{Introduction}
	Distributionally robust optimization (DRO)
	frequently appears in various applications.
	Let $x = (x_1, \ldots, x_n)$ denote the decision vector
	and $\xi = (\xi_1,\ldots, \xi_p)$ be the random vector.
	A typical DRO problem is
	\be \label{md:ecDRO}
	\left\{\begin{array}{cl}
		\min\limits_{x\in X} &  f(x)\\
		\st & \inf\limits_{\mu\in\mathcal{M}}\mathbb{E}_{\mu}[h(x,\xi)]\ge 0,
	\end{array}
	\right.\ee
	where $X\subseteq \re^n$ is a given constraining set for $x$ and
	\[
	f:\re^n\rightarrow \re,\quad
	h:\re^n\times \re^p\rightarrow\re
	\]
	are continuous functions.
	We denote by $f_{min}$ the optimal value of \reff{md:ecDRO}.
	The set $\mc{M}$ is a collection of
	Borel measures whose supports are contained in a prescribed set $S\subseteq \re^p$.
	The inf-constraint is called the robust constraint.
	The $\mathbb{E}_{\mu}[\cdot]$ denotes the expectation
	with respect to the measure $\mu$.
	The $\mc{M}$ is called {\it ambiguity set},
	which describes the uncertainty of the measure $\mu$.
	Typically, the set $\mc{M}$ is described by a set $Y$
	that constrains moments of $\xi$ up to a degree $d$.
	That is, the moment ambiguity set $\mc{M}$ is given as
	\begin{equation}
		\label{eq:mom_ambi}
		\mathcal{M} \,\coloneqq\, \Big \{\mu:
		\supp{\mu}\subseteq S,\,
		\mathbb{E}_{\mu}([\xi]_d)\in Y
		\Big \},
	\end{equation}
	where $\supp{\mu}$ denotes the support of $\mu$,
	and $[\xi]_d$ denotes the monomial vector of $\xi$
	ordered alphabetically with degrees up to $d$.

	It is generally a high challenge to solve DRO globally,
	especially when robust constraints are nonlinear in the decision variables.
	In this paper, we consider \reff{md:ecDRO} such that
	its defining functions are polynomials and
	the moment ambiguity set $\mc{M}$ is given as in \reff{eq:mom_ambi}.
	It is mostly open to solve \reff{md:ecDRO} globally
	when its robust constraints are polynomial (not linear) in $x$.
	In this paper, we give a Moment-SOS relaxation approach
	for solving DRO with polynomial robust constraints.
	Techniques in polynomial optimization are useful for solving DRO problems.
	There are extensive work for polynomial optimization during last two decades.
	A general polynomial optimization problem can be solved globally
	by a hierarchy of Moment-SOS relaxations.
	For interested readers, we also refer to the monographs
	\cite{MomentSOShierarchy,LasBk15,LaurentSOSmom2009,NieBook}
	for detailed introductions to polynomial optimization.

	Distributionally robust optimization has broad applications.
	It is widely used in
	finance \cite{DelageYeDRO},
	network design \cite{DataDrivenPolyDRO,YYDROnetwork},
	inventory problems \cite{DBAdaptiveDRO,XG22} and others.
	There is existing work on solving DRO.
	Approximation methods for solving DRO are studied in \cite{twostageDRO,GS10}.
	DRO with moment ambiguity sets are studied in \cite{Milz,MomentDROXu,JZhangDRO}.
	Some applications and theory are in the monographs \cite{LinFang22,StochasticOpt}.
	Recently, much attention has been paid to DRO problems given by polynomials.
	Polynomial chance-constrained DRO is studied in \cite{LasserreMixAmbiguity}.
	Polynomial constrained DRO is studied in \cite{KlerkDROpolydense,DataDrivenPolyDRO,NieDro2023}
	with applications in newsvendor problem and portfolio optimization.
	We also refer to \cite{ChuToh21,DataDrivenPolyDRO,SZCVaRRobustPortfolio} and references therein
	for more applications of DRO defined with polynomials.

	\subsection*{Contributions}
	
	This article studies DRO defined with
	polynomial robust constraints and moment-based ambiguity sets.
	Consider the optimization problem \reff{md:ecDRO}.
	Suppose $X$ is the semialgebraic set
	\begin{equation}\label{eq:setX}
		X = \{x\in\re^n: c_1(x)\ge 0,\ldots, c_m(x)\ge 0\},
	\end{equation}
	where each $c_i$ is a polynomial function in $x$.
	The moment ambiguity set $\mc{M}$ is given as in \reff{eq:mom_ambi}.
	Let $g = (g_1,\ldots, g_{m_1})$ be a tuple of polynomials.
	Suppose the support $\supp{\mu}$ is contained in the set
	\begin{equation}\label{eq:setS}
		S \, = \, \{\xi\in\re^p: g(\xi)\ge 0\} .
	\end{equation}
	We assume $Y$ is a set given by linear, second-order, and semidefinite constraints.

	It is generally very challenging to solve \reff{md:ecDRO}
	when $h(x,\xi)$ is nonlinear in $x$.
	In the previous work \cite{NieDro2023},
	a Moment-SOS approach is proposed to solve the DRO problem
	when $h(x,\xi)$ is linear in $x$.
	However, when $h(x,\xi)$ is polynomial in $x$,
	the DRO problem is much harder to solve.
	There exists very little work for this case.

	Suppose $d$ is degree of $\xi$ in $h(x,\xi)$.
	We write that
	\[
	h(x,\xi) \,=\, \sum\limits_{\beta\in\N_d^p} h_{\beta}(x)\xi^{\beta} = h(x)^T[\xi]_d,
	\]
	where
	$\N_d^p  \coloneqq  \, \{ \alpha = (\alpha_1, \ldots, \alpha_p)|\, \,
	\alpha_i\in\mathbb{N},\alpha_1+\cdots+\alpha_p\le d \}$
	is the power set of $\xi$ with the highest degree $d$ and
	\[
	[\xi]_{d} \,\coloneqq\,\bbm 1 & \xi_1 & \cdots & \xi_p & \xi_1^2 & \xi_1\xi_2 & \cdots & \xi_p^{d}\ebm^T.
	\]
	In the above, $h(x) \coloneqq (h_{\beta}(x))_{\beta\in\N_d^p}$
	is the coefficient polynomial vector of $h(x,\xi)$ with respect to $\xi$.
	Suppose $2t$ is the smallest even integer that is larger than or equal to
	the degree of $h(x)$ in $x$.
	Then, there is a unique matrix $H$,
	with $\binom{p+d}{d}$ rows and $\binom{n+2t}{2t}$ columns, such that
	\[
	h(x) \,=\, H\cdot[x]_{2t},
	\]
	where $[x]_{2t}\,\coloneqq\, \bbm 1 & x_1 & \cdots & x_n & x_1^2 & x_1x_2 & \cdots & x_n^{2t}\ebm^T$
	is the vector of monomials in $x$ with degrees up to $2t$.
	Let $w = (w_{\alpha})_{\alpha\in\N_{2t}^{n}}$ be a vector variable
	labelled by monomial powers of $x$, where $\N_{2t}^n$ is defined as in \reff{eq:Ndl}.
	If we set $w$
	\[
	w_{\alpha}\,=\, x^{\alpha}\,\coloneqq\, x_1^{\alpha_1}\cdots x_n^{\alpha_n}
	\quad \text{for} \quad
	\alpha = (\alpha_1,\ldots,\alpha_n)\in\N_{2t}^{n},
	\]
	then one can get
	\[
	h(x) = Hw,\quad h(x,\xi) = (Hw)^T[\xi]_d.
	\]
	Hence, the DRO problem \reff{md:ecDRO} is equivalent to
	\begin{equation}  \label{eq:1st:rel}
		\left\{\begin{array}{cl}
			\min\limits_{(x,w)} & f(x)\\
			\st & \inf\limits_{\mu\in\mc{M}} \mathbb{E}_{\mu}[(Hw)^T [\xi]_{d}]\ge 0,\\
			& w = [x]_{2t},\,x\in X,
		\end{array}
		\right.
	\end{equation}
	which has a linear robust constraint in $w$.
	We can solve \reff{eq:1st:rel}
	by the Moment-SOS approach proposed in \cite{NieDro2023}.
	Denote the conic hull
	\begin{equation}\label{eq:coneK}
		K\,\coloneqq\, cone(\{\mathbb{E}_{\mu}([\xi]_d):\mu\in\mc{M}\}).
	\end{equation}
	The robust constraint in \reff{eq:1st:rel} is equivalent to
	\[
	(Hw)^Ty\ge 0,\,\forall y\in K
	\quad\Leftrightarrow\quad
	(Hw)^T[\xi]_d\in K^*,
	\]
	where $K^* = \{q(\xi): vec(q)^Ty\ge 0,\forall y\in K\}$ is the dual cone of $K$,
	and $vec(q)$ denotes the coefficient vector of $q(\xi)$.
	The above right hand side means $(Hw)^T[\xi]_d$, as a polynomial in $\xi$,
	belongs to the cone $K^*$.
	A necessary condition for $w = [x]_{2t}$ with $x\in X$ is that
	\[
	w_{0} = 1,\quad M_t[w]\succeq 0,\quad
	L_{c_i}^{(t)}[w]\succeq 0
	\]
	for every $i \in [m]\coloneqq \{ 1,\ldots, m\}$.
	Here, the $M_t[w]$ is a moment matrix and
	each $L_{c_i}^{(t)}[w]$ is a localizing matrix of $w$
	(see Section~\ref{sc:prelim} for their definitions).
	Hence, \reff{eq:1st:rel} can be relaxed to
	the following linear conic optimization
	\begin{equation}\label{eq:2nd_rel}
		\left\{
		\begin{array}{cl}
			\min & \langle f,w\rangle\\
			\st & (Hw)^T[\xi]_d \in K^*,\\
			& L_{c_i}^{(k)}[w]\succeq 0,\,i=1,\ldots, m, \\
			& M_t[w]\succeq 0,\\
			& w_0 = 1, \, w\in\re^{\N_{2t}^n} .
		\end{array}
		\right.
	\end{equation}
	In the above, the bilinear operation
	$\langle f,w\rangle$ is defined such that
	\begin{equation}\label{eq:bilinearf}
		\langle f,w\rangle  \,\coloneqq\, \sum\limits_{\alpha\in\N_{2t}^n} f_{\alpha} w_{\alpha}\quad
		\mbox{for}\quad f = \sum\limits_{\alpha\in\N_{2t}^n} f_{\alpha} x^{\alpha}.
	\end{equation}
	When $K^*$ is given by linear, second order or semidefinite conic conditions,
	the optimization \reff{eq:2nd_rel} can be solved globally by Moment-SOS relaxations.
	This is discussed in the earlier work \cite{NieDro2023}.
	In particular, \reff{eq:2nd_rel} is a tight relaxation of \reff{md:ecDRO}
	if $f(x), -h(x,\xi)$ and each $-c_i(x)$ are SOS-convex in $x$.
	When \reff{md:ecDRO} is not convex,
	we may still get an optimizer by solving \reff{eq:2nd_rel}.
	However, when $h(x,\xi)$ is nonlinear in $x$,
	the performance of \reff{eq:2nd_rel} for solving
	the DRO \reff{md:ecDRO} globally is studied less in prior existing literature.
	
	We summarize our major contributions as follows.
	
	\begin{itemize}
		\item
		We study DRO with polynomial robust constraints and moment ambiguity sets.
		A moment approach is proposed to transform it into a linear conic optimization relaxation.
		We prove concrete conditions for
		this relaxation to solve the DRO \reff{md:ecDRO} globally.

		\item
		We propose a Moment-SOS approach to solve the DRO globally.
		If the DRO problem is SOS-convex, the proposed approach
		can obtain a global optimizer.
		For nonconvex cases, we give sufficient conditions for the Moment-SOS relaxations to be tight.
		Numerical experiments are given to show the efficiency of our method.
	\end{itemize}

	The rest of the paper is organized as follows.
	In Section~\ref{sc:prelim}, we review some basic results about polynomial optimization.
	In Section~\ref{sec:SolveDRO}, we transform the polynomial DRO
	into linear conic optimization by relaxing the nonlinear robust constraint.
	In Section~\ref{sc:mom_sos}, we propose a Moment-SOS algorithm to solve the DRO.
	In Section~\ref{sec:numerical}, we present some numerical experiments
	to show the efficiency of the method.
	Some conclusions and discussions are made in Section~\ref{sec:con}.
	
	\section{Preliminaries}
	\label{sc:prelim}
	
	\subsection*{Notation}
	The symbol $\mathbb{R}$ (resp., $\mathbb{R}_+$,  $\mathbb{N}$)
	denotes the set of real numbers
	(resp.,  nonnegative real numbers, nonnegative integers).
	For $t\in \mathbb{R}$, $\lceil t\rceil$
	denotes the smallest integer that is greater or equal to $t$.
	For an integer $k>0$, we denote $[k] \coloneqq \{1,\cdots,k\}$.
	The symbol $\mathbb{N}^n$ (resp., $\mathbb{R}^n, \re_+^n$) stands for the set of
	$n$-dimensional vectors with entries in $\mathbb{N}$ (resp., $\mathbb{R}, \re_+$).
	For a vector $v$, we use $\|v\|$ to denote its Euclidean norm.
	The notation $\delta_v$ denotes the atomic Dirac measure supported at $v$.
	Suppose $V$ is a vector space equipped with inner product $\langle \cdot,\cdot\rangle$.
	Let $V^*$ denote the dual space of $V$.
	For a set $S$ in $V$, its closure is denoted as $cl(S)$
	and its interior is denoted as $int(S)$.
	The conic hull of $S$ is
	\[
	cone(S)\,\coloneqq\, \Big\{\sum\limits_{i=1}^N \lambda_i s_i: \lambda_i\in \re_+, s_i\in S, N>0\Big\}.
	\]
	The dual cone of $S$ is given as
	\begin{equation}\label{eq:Sdual}
		S^* \,\coloneqq\, \{q\in V^*: \langle s,q\rangle \ge 0,\, \forall s\in S\}.
	\end{equation}
	We remark that the dual cone of $cone(S)$ is equal to $S^*$ for every $S\subseteq V$.
	The superscript $^T$ denotes the transpose of a matrix or vector.
	The symbol $e$ stands for the vector of all ones
	and $e_i$ is a vector with its $i$th entry being one and all other entries being zeros.
	The $I_n$ denotes the $n$-by-$n$ identity matrix.
	A square matrix $A\in\re^{n\times n}$ is said to be positive semidefinite or psd if
	$x^TAx\ge 0$ for all $x\in\re^n$, denoted as $A\succeq 0$.
	For a function $q$, we use $\nabla q$
	to denote its gradient and $\nabla^2 q$ to denote its Hessian matrix.

	Let $u = x$  or $\xi$ with dimension $l$.
	We use the notation $\mathbb{R}[u]$ to denote
	the ring of polynomials in $u$ with real coefficients.
	The $\mathbb{R}[u]_d$ is the $d$th degree truncation of $\re[u]$.
	If $f$ is a scalar polynomial in $u$, its degree is denoted by $\deg(f)$.
	If $f=(f_1,\ldots, f_m)$ is a tuple of polynomials,
	$\deg(f)$ denotes the highest degree of $\deg(f_i)$ for all $i\in[m]$.
	For $h\in\re[x,\xi]$, we use $\deg_x(h)$ to denote its partial degree in $x$.
	Let $\alpha = (\alpha_1,\ldots, \alpha_l)$ be a power vector of
	$u = (u_1, \ldots, u_l)$.
	We denote the monomial
	\[ u^{\alpha} = u_1^{\alpha_1}\cdots u_l^{\alpha_l}\quad
	\mbox{with}\quad
	|\alpha| \coloneqq \alpha_1+\cdots+\alpha_l.
	\]
	For a degree $d$, denote the power set
	\begin{equation}\label{eq:Ndl}
		\N_d^l \, \coloneqq  \, \{ \af = (\alpha_1, \ldots, \alpha_l)|\, \,
		\alpha_i\in\N,|\alpha|\le d \}.
	\end{equation}

	A polynomial $q\in\re[\xi]$ is said to be a sum-of-squares (SOS) if
	$q = q_1^2+\cdots +q_k^2$ for some $q_i\in\re[\xi]$.
	The SOS polynomial cone in $\xi$ is denoted by $\Sigma[\xi]$.
	Its $d$th degree truncation is
	\[
	\Sig[\xi]_d \, \coloneqq \, \Sig[\xi]\cap \re[\xi]_d .
	\]
	Clearly, each SOS polynomial is nonnegative everywhere,
	but a nonnegative polynomial may not be SOS.
	The approximation quality of standard SOS relaxations
	is studied in \cite{NieSOSbd}.
	A polynomial $\psi(x) \in \re[x]$ is said to be {\it SOS-convex}
	if its Hessian satisfies
	\begin{equation}\label{eq:SOSconv}
		\nabla^2 \psi(x)  = L(x)L(x)^T
	\end{equation}
	for some matrix polynomial $L(x)$.
	Similarly, $\psi(x)$ is said to be {\it SOS-concave} if $-\psi(x)$ is SOS-convex.
	Clearly, every SOS-convex polynomial is convex.
	However, not every convex polynomial is SOS-convex.
	Indeed, {\it every} convex polynomial $\psi\in\re[x_1,\ldots, x_n]_d$
	is SOS-convex if and only if
	$n=1$, or $d=2$, or $(n,d) = (2,4)$.
	This is shown in \cite{AAAPar13}.
	It is NP-hard to determine the convexity
	of polynomials of even degree four or higher \cite{AAA13}.
	For a given polynomial, verifying SOS-convexity can be done by
	solving a semidefinite program.
	By \cite[Lemma~7.1.3]{NieBook}, a polynomial $\psi\in\re[x]_{2k}$
	is SOS-convex if and only if there exists psd matrix
	$X$ of length $n\cdot \binom{n+k-1}{k-1}$ such that
	\begin{equation}\label{eq:sos_verify}
		\nabla^2 \psi(x) = (I_n\otimes [x]_{k-1})^T\cdot X\cdot (I_n\otimes [x]_{k-1}),
	\end{equation}
	where the notation $\otimes$ stands for the classical Kronecker product.
	We refer to \cite[Chapter~7]{NieBook}
	for how to check SOS-convex polynomials.
	When a convex polynomial is not SOS-convex, we can approximate it with
	SOS-convex polynomials \cite{AAAKlerk19},
	which typically requires to use polynomials of higher degrees.
	Similarly, SOS matrix polynomials have broad applications
	in optimization \cite{HilNie08,NiePMI,NieBook}.

	For a subset $S \subseteq \re^p$, we denote
	\[
	\mathscr{P}(S)  \, \coloneqq  \,
	\{q\in\re[\xi]: q(\xi)\ge 0\, \forall \xi \in S \}.
	\]
	For a degree $d$, denote $\mathscr{P}_d(S) \coloneqq \mathscr{P}(S)\cap \re[\xi]_d$.
	The $\mathscr{P}(S)$ and $\mathscr{P}_d(S)$ are closed convex polynomial cones.

	Let $y\in\re^{\N_{d}^p}$ be a truncated multi-sequence (tms) of $\xi$ with degree $d$.
	We define the bilinear operation
	\begin{equation}\label{eq:bilinear}
		\langle q, y\rangle  \coloneqq \sum\limits_{ \alpha\in\N_d^p } q_{\alpha} y_{\alpha}
		\quad \text{for} \quad
		q  = \sum_{ \alpha\in\N_d^p } q_{\alpha} \xi ^{\alpha}.
	\end{equation}
	The tms $y$ is said to admit a measure $\mu$ on $\re^p$ if
	\[
	y_{\alpha} = \int  \xi^{\alpha}\,{\tt d} \mu,\quad
	\forall \alpha\in\N_d^p.
	\]
	In applications, the measure $\mu$ is often required to
	be supported in a set $S \subseteq \re^p$, i.e.,
	$\supp{\mu} \subseteq S$.
	We let $meas(y,S)$ denote the set of measure $\mu$
	that is admitted by $y$ and $\supp{\mu} \subseteq S$.
	Note that when $y$ is the zero tms, $y$ admits the identically
	zero measure, whose support is the empty set.
	Define the cone of all such $y$ as
	\[
	\mathscr{R}_d(S) \, \coloneqq \,
	\{y\in\re^{\N_d^p}: meas(y,S)\not=\emptyset\}.
	\]
	The set $\mathscr{R}_d(S)$ is a convex cone.
	If $S$ is compact, then $\mathscr{R}_d(S)$ is a closed cone.
	When $S$ is not compact,  the cone $\mathscr{R}_d(S)$ may not be closed.
	The cones $\mathscr{P}_d(S)$ and $\mathscr{R}_d(S)$
	are dual to each other and it holds that
	\begin{equation}\label{eq:PRdual}
		\mathscr{P}_d(S) = \mathscr{R}_d(S)^*,\quad
		\mathscr{P}_d(S)^* =  cl(\mathscr{R}_d(S))  .
	\end{equation}
	The above dual relationship follows from the definition \reff{eq:Sdual}
	and the bilinear operation \reff{eq:bilinear}.
	The cone $\mathscr{R}_d(S)$ is generally not self-dual and
	cannot be given by semidefinite programming
	(except some special cases like univariate or quadratic moments).
	If we view $\re^{\N_d^p}$ as a standard Euclidean space
	and the dual relation is given by the Euclidean inner product,
	it is quite complicated to describe the dual cone of $\mathscr{R}_d(S)$.
	However, if we view $\re[\xi]_d$ as the dual space of $\re^{\N_d^p}$
	and the dual relation is given by \reff{eq:bilinear},
	then the dual cone $\mathscr{R}_d(S)^*$ can be conveniently expressed as
	the nonnegative polynomial cone $\mathscr{P}_d(S)$
	and the dual relationships \reff{eq:PRdual} hold.
	The above dual relation between truncated moment sequences
	and polynomials are commonly used in the field of polynomial and moment optimization.
	We refer to the surveys and books \cite{LasBk15,LaurentSOSmom2009,NieBook}
	for more details about this.

	Let $g = (g_1,\ldots, g_{m_1})$ be a tuple of polynomials in $\re[\xi]$.
	The quadratic module of $g$ is the set
	\[
	\qmod{g} \,\coloneqq\, \Sig[\xi] + g_1\cdot \Sig[\xi] + \cdots + g_{m_1}\cdot \Sig[\xi].
	\]
	The $k$th order truncation of $\qmod{g}$ is the subset
	\begin{equation}\label{eq:qmod}
		\qmod{g}_{2k} \,\coloneqq\, \Sig[\xi]_{2k} + g_1\cdot \Sig[\xi]_{2k-\deg(g_1)}
		+ \cdots +g_{m_1}\cdot \Sig[\xi]_{2k-\deg(g_{m_1})}.
	\end{equation}
	If the set $S = \{ \xi \in\re^p: g(\xi)\ge 0\}$, then
	\[
	\qmod{g}_{2k}\subseteq \qmod{g}_{2k+2}\subseteq \cdots
	\subseteq \qmod{g}\subseteq \mathscr{P}(S)
	\]
	for every degree $k$ such that $2k\ge \deg(g)$.
	The $\qmod{g}$ is said to be {\it archimedean}
	if there exists $q \in\qmod{g}$ such that $q(\xi)\ge 0$ gives a compact set.
	When $\qmod{g}$ is archimedean,
	we have $q  \in \qmod{g}$ if $q>0$ on $S$.
	This conclusion is often referenced as
	{\it Putinar's Positivstellensatz}
	(see \cite{NieBook,PutinarPositive}).

	Truncations of quadratic modules and their dual cones
	can be expressed in terms of
	localizing and moment matrices.
	Let $w\in\re^{\N_{2k}^p}$ be a tms of even degree $2k\ge \deg(g)$.
	The $k$th order {\em localizing matrix} of $w$ and $g_i$
	is the symmetric matrix $L_{g_i}^{(k)}[w]$ such that
	\be \label{locmat:gi}
	\vec{b_1}^T\Big(L_{g_i}^{(k)}[w]\Big)\vec{b_2}\,=\,\langle g_i b_1 b_2, w\rangle,
	\ee
	for all polynomials $b_1,b_2 \in \re[\xi]_s$ with $s= k-\lceil \deg(g_i)/2\rceil$.
	In the above,
	$\vec{b_1},\vec{b_2}$ denote the coefficient vectors of $b_1,b_2$, respectively.
	For the special case that $g_i = 1$ is the identically one polynomial,
	\begin{equation}\label{eq:mommat}
		M_k[w]\,\coloneqq\, L_1^{(k)}[w]
	\end{equation}
	is called the $k$th order {\it moment matrix} of $w$.

	For instance, for the polynomial tuple
	$g = (\xi_1, \xi_1- \xi_2^2)$, we have
	\[
	M_2[w] = \begin{bmatrix}
		w_{00} & w_{10} & w_{01} & w_{20} & w_{11} & w_{02}\\
		w_{10} & w_{20} & w_{11} & w_{30} & w_{21} & w_{12}\\
		w_{01} & w_{11} & w_{02} & w_{21} & w_{12} & w_{03}\\
		w_{20} & w_{30} & w_{21} & w_{40} & w_{31} & w_{22}\\
		w_{11} & w_{21} & w_{12} & w_{31} & w_{22} & w_{13}\\
		w_{02} & w_{12} & w_{03} & w_{22} & w_{13} & w_{04}
	\end{bmatrix}, \quad
	L_{\xi_1}^{(2)}[w] = \begin{bmatrix}
		w_{10} & w_{20} & w_{11}\\
		w_{20} & w_{30} & w_{21}\\
		w_{11} & w_{21} & w_{12}
	\end{bmatrix},
	\]
	\[
	L_{\xi_1-\xi_1^2}^{(2)}[w] = \begin{bmatrix}
		w_{10}-w_{02} & w_{20}-w_{12} & w_{11}-w_{03}\\
		w_{20}-w_{12} & w_{30}-w_{22} & w_{21}-w_{13}\\
		w_{11}-w_{03} & w_{21}-w_{13} & w_{12}-w_{04}
	\end{bmatrix}.
	\]

	The dual cone of the truncation $\qmod{g}_{2k}$  is the set
	\[
	\mathscr{S}[g]_{2k} \,\coloneqq\, \big\{w\in\re^{\N_{2k}^l}:
	M_{k}[w]\succeq 0, L_{g_1}^{(k)}[w]\succeq 0,\ldots,
	L_{g_{m_1}}^{(k)}[w]\succeq 0
	\big\}.
	\]
	In other words, it holds
	\begin{equation}\label{eq:qmdual}
		(\qmod{g}_{2k})^* \,=\, \mathscr{S}[g]_{2k}
	\end{equation}
	for all degrees $2k\ge \deg(g)$.
	We refer to \cite[Chapter~2]{NieBook} for this.
	For recent advances in polynomial optimization
	and truncated moment problems, we refer to
	\cite{MomentSOShierarchy,LasBk15,LaurentSOSmom2009,NieSOSbd,Tight18,NieBook}
	and references therein.
	Interestingly, moment relaxations are also
	very useful in solving tensor optimization
	\cite{njwSTNN17,NieZhang18}.

	\section{Relaxations for the polynomial robust constraint}
	\label{sec:SolveDRO}

	We consider the polynomial DRO problem \reff{md:ecDRO}.
	Suppose $d$ is the degree of $h(x,\xi)$ in
	$\xi = (\xi_1, \ldots, \xi_p)$ and
	\[
	h(x,\xi) \,=\, \sum\limits_{\beta\in\N_d^p}
	h_{\beta}(x)\xi^{\beta} = h(x)^T[\xi]_d,
	\]
	where $h(x)$ is the polynomial coefficient vector with respect to $\xi$.
	Denote the degree
	\begin{equation}\label{eq:td}
		t \,\coloneqq\, \max \big\{\left\lceil \deg_x(h)/2\right\rceil,
		\lceil\deg(f)/2\rceil, \lceil \deg(c)/2\rceil \big\},
	\end{equation}
	where $c = (c_1,\ldots, c_m)$.
	There is a unique matrix $H$, with $\binom{p+d}{d}$ rows and $\binom{n+2t}{2t}$ columns,
	such that
	\begin{equation}\label{eq:H[x]}
		h(x) = H\cdot [x]_{2t}.
	\end{equation}
	For instance, if $n = p = 2$ and
	\[
	h(x,\xi) = (1+x_1)^2+ x_1\xi_1+x_2\xi_2+(x_1^2+x_2)\xi_1^2+2x_1x_2\xi_1\xi_2+(x_1+x_2^2)\xi_2^2 ,
	\]
	then $t = 1$ and the matrix $H$ is given as below:
	\[
	h(x) = \bbm (1+x_1)^2\\ x_1\\x_2\\x_1^2+x_2\\2x_1x_2\\x_1+x_2^2\ebm
	=\underbrace{\bbm 1 & 2 & 0 & 1 & 0 & 0 \\
		0 & 1 & 0 & 0 & 0 & 0 \\
		0 & 0 & 1 & 0 & 0 & 0\\
		0 & 0 & 1 & 1 & 0 & 0\\
		0 & 0 & 0 & 0 & 2 & 0\\
		0 & 1 & 0 & 0 & 0 & 1\ebm}_H \cdot
	\underbrace{\bbm 1\\x_1\\x_2\\x_1^2\\x_1x_2\\x_2^2\ebm}_{[x]_2}.
	\]

	\subsection{Conic representation of robust constraints}
	\label{ssc:conic:rob}
	
	Let $w = (w_{\alpha})_{\alpha\in \N_{2t}^{n}}$
	be a tms variable and $H$ be the matrix as in \reff{eq:H[x]}.
	If $w = [x]_{2t}$ for some $x\in \re^n$,
	then by \reff{eq:H[x]}, we have
	\[
	h(x,\xi)\,=\,(Hw)^T[\xi]_d.
	\]
	The point $x$ is feasible for \reff{md:ecDRO}
	if and only if $x \in X$ and
	\be  \label{eq:HwE>=0}
	(Hw)^T\mathbb{E}_{\mu}([\xi]_{d}) \ge 0
	\quad \text{for all} \quad  \mu\in\mc{M}.
	\ee
	The above constraint can be reformulated
	in terms of dual conic conditions.
	Let $K$ be the conic hull
	\[
	K \,\coloneqq\,  cone(\{\mathbb{E}_{\mu}([\xi]_{d}):\mu\in \mc{M}\}).
	\]
	Then \reff{eq:HwE>=0} is equivalent to the dual cone membership
	\[
	(Hw)^T[\xi]_d \in K^*.
	\]
	Note that the dual cone $K^*$ is a subset of $\re[\xi]$.
	Assume the moment ambiguity set $\mc{M}$
	is given as in  \reff{eq:mom_ambi}, so
	\be \label{eq:coneK:meas}
	K \,=\, \mathscr{R}_{d}(S)\cap cone(Y),
	\ee
	where  $cone(Y)$ is the conic hull of $Y$ and
	\[
	\mathscr{R}_{d}(S) = \{y\in\re^{\N_d^p}: meas(y, S)\not=\emptyset\} .
	\]
	As introduced in Section~\ref{sc:prelim},  it holds that
	\[
	\mathscr{P}_d(S)\,=\,\mathscr{R}_d(S)^*,\quad
	Y^* \,=\, cone(Y)^*.
	\]
	If $K = cl(\mathscr{R}_d(S))\cap cl(cone(Y))$, its dual cone can be expressed as
	\begin{equation}\label{eq:Kdual}
		K^* \,=\, \mathscr{P}_{d}(S) + Y^*.
	\end{equation}
	This is because $K = (K^*)^*$ when $K$ is a closed convex cone.
	The equation \reff{eq:Kdual} is often satisfied.
	For instance, when $\mathscr{R}_d(S), cone(Y)$ are closed and
		their interiors have non-empty intersections \cite[Proposition B.2.7]{BenNem12}.
	The cones $\mathscr{R}_d(S)$ and $cone(Y)$ are closed when $S,Y$ are compact sets.

	An interesting fact is that $\mathscr{P}_d(S)$
	has explicit semidefinite expression when
	$\xi$ is univariate and $S$ is an interval.
	For this case, $K^*$ also have semidefinite expressions.

	\begin{example}\label{ex:Kdual}
		Suppose $\xi$ is univariate, $S = [0,1]$, and $Y\subseteq \re^4$ is given by
		\[
		y_0 = 1,\quad
		Ay = \left[\begin{array}{rrrr} 1 & -1 & 0 & 0\\
			0 & 1 & -2 & 0\\ 0 & 0 & 2 & -3\\
			0 & 0 & 0 & 3\end{array}\right]\bbm y_0\\y_1\\y_2\\y_3\ebm\ge 0.
		\]
		Then, \reff{eq:Kdual} holds as $S$ and $Y$ are both compact.
		Since $\xi$ is univariate and
		$S = \{\xi: \xi \ge 0, \, 1-\xi \ge 0\}$,
		we have (see \cite[Chapter~3]{NieBook})
		\[
		\mathscr{P}_3([0,1]) = \xi\cdot \Sig[\xi]_2 + (1-\xi)\cdot\Sig[\xi]_2.
		\]
		Since $cone(Y) = \{y\in\re^4: Ay\ge 0\}$, we get
		\[
		Y^* \quad =\quad \left\{ q^T[\xi]_3: q = (q_0,q_1,q_2,q_3)^T,\, q= A^Tu,\,  u\in\re_+^4 \right\}
		\]
		\[\,=\, \left\{q_0+q_1\xi+q_2\xi^2+q_3\xi^3
		\left|\begin{array}{l}
			q_0\ge 0,\quad q_0+q_1\ge 0,\\
			2q_0+2q_1+q_2\ge 0,\\
			3q_0+3q_1+1.5q_2+q_3\ge 0
		\end{array}
		\right.
		\right\}.
		\]
	\end{example}

	\bigskip
	Suppose $K^* = \mathscr{P}_d(S)+Y^*$.
	Then \reff{md:ecDRO} can be equivalently given as
	\begin{equation}\label{eq:po_equiv}
		\left\{\begin{array}{cl}
			\min\limits_x & f(x)\\
			\st & \big(H [x]_{2t}\big)^T[\xi]_d \in \mathscr{P}_{d}(S)+Y^*,\\
			& x\in X.
		\end{array}
		\right.
	\end{equation}
	Denote the projection operator $\pi: \re^{\N_{2t}^n}\to \re^n$ such that
	\begin{equation}\label{eq:pi}
		\pi(w) \,\coloneqq\, \big( w_{e_1},\ldots, w_{e_n} \big).
	\end{equation}
	A necessary condition for $w = [x]_{2t}$ is that
	\[
	w_0 = 1,\quad x = \pi(w),\quad M_{t}[w] \succeq 0.
	\]
	We refer to \reff{eq:mommat}
	for the moment matrix $M_{t}[w]$.
	So \reff{eq:po_equiv} can be relaxed to
	\be  \label{eq:momrel}
	\left\{\begin{array}{cl}
		\min\limits_{ (x, w)} & f(x)\\
		\st & (Hw)^T[\xi]_d \in \mathscr{P}_{d}(S)+Y^*, \\
		& x = \pi(w) \in X, \,\,  M_{t}[w]\succeq 0,\\
		& w_0 = 1,\,\,  w\in\re^{\N_{2t}^n} .
	\end{array}
	\right.
	\ee
	The following gives a condition for \reff{eq:momrel}
	to be a tight relaxation for \reff{eq:po_equiv}.
	For this purpose, we introduce Jensen's inequality for SOS-convex polynomials:
	if $\varphi\in\re[x]_{2t}$ is SOS-convex and
	$w\in\re^{\N_{2t}^n}$ is a tms such that $w_0 = 1$ and $M_t[w]\succeq 0$, then
	\be \label{ineq:Jensen}
	\varphi(\pi(w))\le \langle \varphi,w\rangle,
	\ee
	where the bilinear operation $\langle \varphi,w\rangle$ is defined as in \reff{eq:bilinearf}.
	We refer to \cite{Las09,NieBook} for the proof.
	\begin{theorem}\label{thm:eqv1}
		Suppose \reff{eq:Kdual} holds and $\mathbb{E}_{\mu}[h(x,\xi)]$
		is SOS-concave in $x$ for each $\mu\in \mc{M}$.
		Then, \reff{eq:momrel} is a tight relaxation for \reff{md:ecDRO},
		in the following sense: $x$ is an optimizer of \reff{md:ecDRO}
		if and only if $(x, w)$ is an optimizer of \reff{eq:momrel}.
	\end{theorem}
	\begin{proof}
		When \reff{eq:Kdual} holds, \reff{md:ecDRO} is equivalent to \reff{eq:po_equiv}.
		If $x$ is a feasible point of \reff{eq:po_equiv},
		then $w = [x]_{2t}$ must also be feasible for \reff{eq:momrel}.
		Conversely, if $(x,w)$ is a feasible point of \reff{eq:momrel}, then
		\[
		\mathbb{E}_{\mu}[(Hw)^T[\xi]_d]\ge 0 \,\, \forall \mu\in\mc{M}.
		\]
		Fix an arbitrary $\mu\in \mc{M}$, and let
		\[ \varphi(x) \,\coloneqq\,  -\mathbb{E}_{\mu}[h(x,\xi)]
		\,=\, -h(x)^T\mathbb{E}_{\mu}([\xi]_d). \]
		Since $\varphi$ is SOS-convex and $w$ is feasible for \reff{eq:momrel},
		Jensen's inequality \reff{ineq:Jensen} implies
		\[
		-\mathbb{E}_{\mu}[h(\pi(w),\xi)] \,=\, \,   \varphi(\pi(w))\le
		\langle \varphi,w\rangle \, \,=\, -\mathbb{E}_{\mu}[\langle Hw, [\xi]_d\rangle] ,
		\]
		so it holds
		\[
		\mathbb{E}_{\mu}[h(\pi(w),\xi)]\, \ge\, \mathbb{E}_{\mu}[\langle Hw, [\xi]_d\rangle ].
		\]
		Since the above inequality holds for arbitrary $\mu\in\mc{M}$, we get
		\[
		\mathbb{E}_{\mu}[h(\pi(w), \xi)] \,\ge\,
		\mathbb{E}_{\mu}[\langle Hw,[\xi]_d\rangle] \,\ge\, 0
		\quad \forall \mu\in \mc{M}.
		\]
		So $x = \pi(w)$ is feasible for \reff{eq:po_equiv}
		whenever $(x,w)$ is feasible for \reff{eq:momrel}.
		Since they have the same objective,
		\reff{eq:momrel} is a tight relaxation for \reff{md:ecDRO}.
	\end{proof}
	When the distribution of $\mu$ is unknown, a sufficient condition for
	$\mathbb{E}_{\mu}[h(x,\xi)]$ to be SOS-concave is that
	$h(x,\xi)$ is SOS-concave in $x$ for every $\xi\in S$, i.e.,
	\[ -\nabla_x^2 h(x,\xi) \,=\, L_{\xi}(x)L_{\xi}(x)^T\succeq 0,\quad \forall \xi\in S, \]
	where $L_{\xi}(x)$ is a polynomial matrix in $x$ with coefficients
	depending on $\xi$.
	Then,
	\[ -\mathbb{E}_{\mu}[\nabla_x^2 h(x,\xi)]
	\,=\, \mathbb{E}_{\mu}\big[ L_{\xi}(x)L_{\xi}(x)^T\big] =
	\int_{S} L_{\xi}(x)L_{\xi}(x)^T{\tt d}\mu(\xi) \]
	is an SOS matrix polynomial in $x$,
	since the cone of SOS matrix polynomials (for a fixed degree) is closed.
	We refer to \cite[Chapter~7\&10]{NieBook} for this fact.

	The problem~\reff{eq:momrel} is linear conic optimization
	when $f$ is linear and $X$ is given by linear constraints.
	Here is an example.

	\begin{example}\label{ex:uni_xi_lin}
		Consider the DRO problem
		\begin{equation}\label{eq:uni_xi_lin}
			\left\{
			\begin{array}{cl}
				\min\limits_{x\in\re^2} &  x_1 - 2x_2\\
				\st & \inf\limits_{\mu\in\mc{M}}\,\mathbb{E}_{\mu}
				[1+x_1\xi-2x_2\xi^2+(x_1-x_2^2)\xi^3]\ge 0,\\
				& x\ge 0,\quad 1-e^Tx\ge 0,
			\end{array}
			\right.
		\end{equation}
		where $\xi\in S=[0,1]$ is a univariate random variable
		and $\mc{M}$ is given as in Example~\ref{ex:Kdual}.
		For $t = 1$ and $d = 3$,
		the problem \reff{eq:momrel} is a semidefinite program
		\[
		\left\{
		\begin{array}{cl}
			\min\limits_{w} & w_{10} - 2w_{01}\\
			\st
			& \bbm w_{00}\\w_{10}\\-2w_{01}\\w_{10}-w_{02}\ebm^T
			\left[\begin{array}{l} 1\\ \xi\\ \xi^2\\ \xi^3\end{array}\right]
			\in \mathscr{P}_3([0,1])+Y^*,\\
			&  \bbm w_{00} & w_{10} & w_{01}\\ w_{10} & w_{20} & w_{11}\\
			w_{01} & w_{11} & w_{02}\ebm \succeq 0,\\
			& w_{10}\ge 0,\, w_{01}\ge 0, 1-w_{10}-w_{01}\ge 0, \\
			& w_{00} = 1, \,\, w\in\re^{\N_2^2} .
		\end{array}
		\right.
		\]
		where $Y^*$ is given explicitly in Example~\ref{ex:Kdual}.
		Its optimizer is
		\[
		w^* \,=\, (1, 0, 1, 0, 0, 1).
		\]
		The $h(x,\xi)$ is SOS-concave in $x$ for every $\xi\in S = [0,1]$, since
		\[ -\nabla_x^2h(x,\xi) = \bbm 0 &0\\0 & 2\xi^3\ebm = 2\xi
		\bbm 0\\ \xi\ebm\bbm 0\\ \xi\ebm^T\succeq 0. \]
		Then $\mathbb{E}_{\mu}[h(x,\xi)]$ is SOS-concave for all $\mu\in\mc{M}$,
		since $-\mathbb{E}_{\mu}[\nabla_x^2h(x,\xi)]$ is a constant psd matrix.
		By Theorem~\ref{thm:eqv1}, a global optimizer of \reff{eq:uni_xi_lin} is
		$x^* = \pi(w^*) \,=\, (0, 1).$
		Hence this DRO has the global minimum $f_{min} = f(x^*) = -2$.
	\end{example}
	
	\subsection{Moment relaxations for functions in $x$}

	When $f$ or any $c_i$ is nonlinear,
	we can further relax \reff{eq:momrel}
	to a linear conic optimization problem.
	Recall that
	\[
	X = \{x\in\re^n: c_1(x)\ge 0, \ldots, \, c_m(x) \ge 0  \}.
	\]
	Let $t$ be the degree as in \reff{eq:td}.
	If $w = [x]_{2t}$ for some $x\in X$, then
	\[
	f(x) = \langle f, w\rangle,\quad
	L_{c_i}^{(t)}[w]\succeq 0\, (i\in[m]).
	\]
	This leads to the moment relaxation
	\be  \label{eq:momrel1}
	\left\{\begin{array}{cl}
		\min & \langle f, w\rangle\\
		\st
		& (Hw)^T [\xi]_d \in \mathscr{P}_{d}(S)+Y^*,\\
		& L_{c_i}^{(t)}[w]\succeq 0,\,i\in[m], \\
		& M_t[w]\succeq 0,\\
		& w_0 = 1, \, w\in \re^{\N_{2t}^n}.\\
	\end{array}
	\right.
	\ee
	Interestingly, \reff{eq:momrel1} and \reff{eq:momrel}
	are equivalent if $f$ and each $-c_i$ are SOS-convex polynomials.
	Let $f^*$ be the optimal value of \reff{eq:momrel1}.
	Clearly, $f^*$ is a lower bound of the minimum value $f_{min}$ of \reff{md:ecDRO}.

	\begin{theorem}\label{thm:sosproj}
		Suppose $f, -c_1,\ldots, -c_m$ are SOS-convex,
		then \reff{eq:momrel} and \reff{eq:momrel1}
		have the same optimal value and $w^*$ is an optimizer of \reff{eq:momrel1}
		if and only if $x^*\coloneqq \pi(w^*)$ is an optimizer of \reff{eq:momrel}.
		Moreover, when \reff{eq:Kdual} holds and $\mathbb{E}_{\mu}[h(x,\xi)]$
		is SOS-concave for every $\mu\in\mc{M}$,
		\reff{eq:momrel1} is a tight relaxation of \reff{md:ecDRO}
		in the sense as in Theorem~\ref{thm:eqv1}.
	\end{theorem}
	\begin{proof}
		Let $w$ be a feasible point of \reff{eq:momrel1}.
		When each $c_i$ is SOS-concave, we have
		\[
		c_i(\pi(w))\,\ge\, \langle c_i,w\rangle \,\ge\, 0,\quad\forall i\in[m].
		\]
		In the above, the first inequality is due to
		Jensen's inequality \reff{ineq:Jensen}
		and the second one is implied by $L_{c_i}^{(t)}[w]\succeq 0$.
		So $x=\pi(w)$ is feasible for \reff{eq:momrel} if
		$w$ is feasible for \reff{eq:momrel1}.
		Let $\hat{f},f^*$ denote the optimal values of
		\reff{eq:momrel} and  \reff{eq:momrel1}, respectively.
		Clearly, $\hat{f}\ge f^*$.
		Suppose $w^*$ is an optimizer of \reff{eq:momrel1}.
		Then $x^*=\pi(w^*)$ is feasible for \reff{eq:momrel}, thus
		\[
		f(\pi(w^*))\,\ge\, \hat{f}\, \ge \, f^* \,=\, \langle f,w^*\rangle.
		\]
		On the other hand, since $f$ is SOS-convex, we have
		\[
		f(\pi(w^*)) \,\le\, \langle f,w^*\rangle
		\]
		by Jensen's inequality \reff{ineq:Jensen} again.
		
		When \reff{eq:Kdual} holds and $\mathbb{E}_{\mu}[h(x,\xi)]$
		is SOS-concave for every $\mu\in\mc{M}$, the last statement
		follows from Theorem~\ref{thm:eqv1}.
	\end{proof}

	The following is an example to illustrate Theorem~\ref{thm:sosproj}.
	
	\begin{example}\label{ex:uni_xi_nol}
		Consider the DRO problem
		\begin{equation}\label{eq:uni_xi_nol}
			\left\{
			\begin{array}{cl}
				\min\limits_{x\in\re^2} &  2x_1-3x_2+x_1^2-x_1x_2+x_2^2\\
				\st & \inf\limits_{\mu\in\mc{M}}\,\mathbb{E}_{\mu}
				[(x_2-x_1^2)\xi+x_1x_2\xi^2+(x_1-x_2^2)\xi^3]\ge 0,\\
				& 1-x_1^2\ge 0,\, 1-x_2^2\ge 0,
			\end{array}
			\right.
		\end{equation}
		where $\xi\in S = [0,1]$ is univariate and $\mc{M}$
		is given as in Example~\ref{ex:Kdual}.
		It is easy to verify that the objective polynomial is SOS-convex
		and all constraining polynomials are SOS-concave (for every $\xi\in S=[0,1]$), since
		\[
		\nabla^2f(x) = \left[\begin{array}{rr} 2 & -1\\-1 & 2\end{array}\right]\succeq 0,\quad
		-\nabla_x^2 h(x,\xi) = \xi\left[\begin{array}{rr}
			2 & -\xi\\ -\xi & 2\xi^2\end{array}\right]\succeq 0.
		\]
		For $t = 1$ and $d=3$, the relaxation \reff{eq:momrel1} is a semidefinite program:
		\[
		\left\{
		\begin{array}{cl}
			\min\limits_{w} & 2w_{10}-3w_{01}+w_{20}-w_{11}+w_{02}\\
			\st
			& w\in\re^{\N_2^2},\,w_{00} = 1,\,\\
			& \bbm 0\\ w_{01}-w_{20}\\ w_{11}\\ w_{10}-w_{02}\ebm^T
			\left[\begin{array}{l} 1\\ \xi\\ \xi^2\\ \xi^3\end{array}\right]
			\in Y^*+\mathscr{P}_3([0,1]),\\
			&  \bbm w_{00} & w_{10} & w_{01}\\ w_{10} & w_{20} & w_{11}\\
			w_{01} & w_{11} & w_{02}\ebm \succeq 0,\\
			& w_{00}-w_{20}\ge 0,\, w_{00}-w_{02}\ge 0.
		\end{array}
		\right.
		\]
		Its optimal value and optimizer are respectively
		\[
		f^* = -\frac{9}{4},\quad
		w^* = \frac{1}{4}  (4,\, -2,\, 4,\, 1,\, -2,\, 4).
		\]
		By Theorem~\ref{eq:uni_xi_lin}, \reff{eq:uni_xi_nol} has
		a global minimizer
		$x^* =  \pi(w^*) = \frac{1}{2} (-1, 2)$
		and the global minimum $f_{min} = f^* = -\frac{9}{4}$.
	\end{example}

	\section{A Moment-SOS Algorithm}
	\label{sc:mom_sos}

	In this section, we propose a Moment-SOS algorithm for solving
	the optimization problem~\reff{eq:momrel1},
	which is a relaxation of DRO \reff{md:ecDRO}
	with polynomial robust constraints.
	In \reff{eq:momrel1}, the $\mathscr{P}_d(S)$
	is the cone of polynomials in $\re[\xi]_d$
	that are nonnegative on $S$.
	Usually, the cone $\mathscr{P}_d(S)$ does not have
	a computationally convenient expression.
	Assume the set $S$ is given as
	\begin{equation}
		\label{eq:S}
		S\,\coloneqq\, \{\xi\in\re^p: g_1(\xi)\ge 0,\ldots, g_{m_1}(\xi)\ge 0\},
	\end{equation}
	for a polynomial tuple $g = (g_1,\ldots, g_{m_1})$ in $\xi$.
	Let
	\begin{equation}\label{eq:d0}
		d_0\,\coloneqq\, \max\big\{ \lceil d/2\rceil, \lceil \deg(g)/2\rceil \big\}.
	\end{equation}
	For each degree $k\ge d_0$, the truncations of $\qmod{g}_{2k}$
	(see \reff{eq:qmod} the notation) satisfies
	\[
	\qmod{g}_{2k}\cap \re[\xi]_d \,\subseteq \,
	\qmod{g}_{2k+2}\cap \re[\xi]_d\,\subseteq \cdots\subseteq\, \mathscr{P}_d(S).
	\]
	In particular, when $\qmod{g}$ is archimedean,
	it holds that (see \cite[Chapter~8]{NieBook})
	\[
	\mbox{int} \Big( \mathscr{P}_d(S) \Big) \, = \,
	\bigcap_{ k \ge d_0}  \qmod{g}_{2k}\cap \re[\xi]_d .
	\]
	So, we consider the $k$th order SOS approximation for
	\reff{eq:momrel1}:
	\begin{equation}  \label{eq:mom_sos}
		\left\{\begin{array}{rl}
			f_k\,\coloneqq\,\min & \langle f, w\rangle\\
			\st
			& (Hw)^T [\xi]_d \in (\qmod{g}_{2k}\cap \re[\xi]_d)+ Y^*,\\
			& L_{c_i}^{(t)}[w]\succeq 0,\,i = 1, \ldots, m, \\
			& M_t[w]\succeq 0,\\
			&  w_0 = 1, \,  w\in\re^{\N_{2t}^n} .\\
		\end{array}
		\right.
	\end{equation}
	Consider the Lagrangian function:
	\[\begin{array}{rcl}
		\mc{L}(w; \gamma, y, q) & \coloneqq & \langle f, w\rangle - \gamma(w_0-1) -
		\langle Hw, y\rangle -\langle w, q\rangle \\
		& =& \langle f,w\rangle -\gamma w_0-\langle H^Ty,w\rangle -\langle q, w\rangle + \gamma\\
		&=& \langle f - H^Ty - q - \gamma\cdot 1, w\rangle +\gamma,
	\end{array}\]
	where $\gamma\in\re$, $q\in \qmod{c}_{2k}$, and $y\in cone(Y)$
	is the $d$th order truncation of a tms in $\mathscr{S}[g]_{2k}$.
	In the above,
	$-\gamma w_0 = \langle-\gamma\cdot 1, w\rangle$ follows from the bilinear operation \reff{eq:bilinearf}
	for the constant scalar polynomial $-\gamma$, and
	the $f-H^Ty-q-\gamma$ denotes the polynomial
	\[
	f(x) - y^TH\cdot [x]_{2t} -q(x) - \gamma\in\re[x].
	\]
	To make $\mc{L}(w; \gamma, y, q)$ be bounded below for all $w\in\re^{\N_{2t}^n}$,
	the above polynomial must be identically zero, or equivalently,
	\[
	f(x) - y^TH\cdot [x]_{2t} - \gamma \,=\, q(x) \in\qmod{c}_{2t}.
	\]
	When $cone(Y)$ is closed, the dual problem of \reff{eq:mom_sos} is
	\begin{equation}\label{eq:polydual}
		\left\{
		\begin{array}{rl}
			\gamma_{k}\,\coloneqq\,\max\limits_{(\gamma, y, z)} & \gamma\\
			\st & f(x) - y^TH\cdot [x]_{2t} - \gamma\in\qmod{c}_{2t},\\
			& \gamma\in\re,\,y\in cone(Y),\\
			& y = z|_d,\, M_{k}[z]\succeq 0,\\
			& L_{g_i}^{(k)}[z]\succeq 0\,(i\in[m_1]),
		\end{array}
		\right.
	\end{equation}
	where $z|_d$ denotes the degree-$d$ truncation of $z$ as
	\[
	z|_d \, \coloneqq \,  (z_\af)_{  \af\in \N^p_d }.
	\]
	The primal-dual pair \reff{eq:mom_sos}-\reff{eq:polydual}
	is called the $k$th order Moment-SOS relaxation for \reff{eq:momrel1}.
	When $Y^*$ is a mixture of linear, second-order and semidefinite conic conditions,
	the pair \reff{eq:mom_sos}-\reff{eq:polydual}
	can be solved as a semidefinite program.
	
	Assume the equality \reff{eq:Kdual} holds.
	We give the following Moment-SOS algorithm for
	solving \reff{eq:momrel1}
	which is a relaxation of DRO \reff{md:ecDRO}.

	\begin{algorithm}
		\label{def:alg}
		For the DRO given as in \reff{md:ecDRO},
		let $t$ be as in \reff{eq:td},
		$d_0$ as in \reff{eq:d0}.
		Initialize $k \coloneqq d_0$, $k_1 \coloneqq  \lceil (d+1)/2 \rceil$
		and do the following:
		
		\begin{description}
			
			\item[Step 0]
			Formulate the cone $\cone{Y}$ and select a generic $R\in\Sig[\xi]_{2k_1}$.
			
			\item[Step 1]
			Solve the $k$th order Moment-SOS relaxations
			\reff{eq:mom_sos}--\reff{eq:polydual}
			for their optimizer $w^*$ and $(\gamma^*,y^*, z^*)$ respectively.

			\item[Step~2]
			Solve the moment optimization problem
			\begin{equation}\label{eq:TKMP}
				\left\{
				\begin{array}{cl}
					\min\limits_{\hat{z}} & \langle R, \hat{z}\rangle\\
					\st
					& L_{g_i}^{(k_1)}[\hat{z}]\succeq 0,\,i=1, \ldots, m_1, \\
					& M_{k_1}[\hat{z}]\succeq 0,\,  \hat{z}|_d = y^*,\\
					& \hat{z}\in\re^{\N_{2k_1}^p} .
				\end{array}
				\right.
			\end{equation}
			If \reff{eq:TKMP} is infeasible, update $k\coloneqq k+1$ and go back to Step~1.
			If \reff{eq:TKMP} is feasible, solve it for
			an optimizer $\hat{z}^*$ and go to Step~3.
			
			\item[Step~3]
			Check if there exists a degree $d_1\in [d_0,k_1]$ such that
			\be  \label{eq:flat}
			\rank\, M_{d_1}[\hat{z}^*] \,=\, \rank\, M_{d_1-d_2}[\hat{z}^*],
			\ee
			where $d_0$ is given as in \reff{eq:d0} and
			$d_2\coloneqq \lceil \deg(g)/2\rceil$.
			If it does exist, go to Step~4; otherwise,
			update $k_1 \coloneqq k_1+1$ and go back to Step~2.
			
			\item[Step 4]
			Output $x^* = \pi(w^*)$.
		\end{description}
	\end{algorithm}
	
	The following flowchart describes the process of
      Algorithm~\ref{def:alg}, each approximation step of \reff{md:ecDRO},
      and the corresponding theorems.
		\begin{figure}[!ht]
		\centering
		\resizebox{1\textwidth}{!}{%
			\begin{circuitikz}
				\tikzstyle{every node}=[font=\Large]
				\draw  (2.25,15) rectangle  node {\LARGE (1.1)} (4,14);
				\draw [<->, >=Stealth] (4,14.5) -- (7,14.5);
				\draw  (7,15) rectangle  node {\LARGE (3.6)} (8.75,14);
				\node [font=\normalsize] at (5.5,14.75) {Reformulation};
				\draw [->, >=Stealth] (8.75,14.5) -- (12,14.5);
				\draw  (12,15) rectangle  node {\LARGE (3.8)} (13.75,14);
				\node [font=\normalsize] at (5.5,14.25) {Condition (3.5)};
				\node [font=\normalsize] at (10.25,14.25) {Theorem 3.2};
				\node [font=\normalsize] at (10.25,14.75) {Tight Relaxation};
				\draw [->, >=Stealth] (13.75,14.5) -- (16.75,14.5);
				\draw  (16.75,15) rectangle  node {\Large (3.11)} (18.5,14);
				\node [font=\normalsize] at (15.25,14.75) {Tight Relaxation};
				\node [font=\normalsize] at (15.25,14.25) {Theorem 3.4};
				\node [font=\normalsize] at (12.25,15.75) {Moment relaxation with respect to $x$};
				\draw  (12,13) rectangle  node {\LARGE $(4.3)$} (13.75,12);
				\draw  (9.25,13) rectangle  node {\LARGE $(4.4)$} (11,12);
				\draw [short] (7.75,15) -- (7.75,16);
				\draw [short] (7.75,16) -- (17.75,16);
				\draw [->, >=Stealth] (17.75,16) -- (17.75,15);
				\node [font=\normalsize] at (11.5,12.75) {dual};
				\draw [short] (17.75,14) -- (17.75,12.5);
				\draw [->, >=Stealth] (17.75,12.5) -- (14.75,12.5);
				\node [font=\normalsize] at (16.25,12.75) {SOS approximation };
				\node [font=\normalsize] at (13.75,12.75) {$$};
				\node [font=\normalsize] at (16.25,12.25) {Theorem 4.2};
				\node [font=\normalsize] at (8.75,10.75) {\textbf{Algorithm 4.1}};
				\node [font=\normalsize] at (11.5,11.5) {\textbf{Step 1}};
				\draw  (6.25,13) rectangle  node {\LARGE (4.5)} (8,12);
				\draw  (3.25,13) rectangle  node {\LARGE (4.6)} (5,12);
				\draw [ dashed] (3,13.25) rectangle  (5.25,11.75);
				\draw [ dashed] (6,13.25) rectangle  (8.25,11.75);
				\draw [<->, >=Stealth] (11,12.5) -- (12,12.5);
				\draw [ dashed] (9,13.25) rectangle  (14,11.75);
				\draw [->, >=Stealth, dashed] (9,12.5) -- (8.25,12.5);
				\draw [->, >=Stealth, dashed] (6,12.5) -- (5.25,12.5);
				\node [font=\normalsize] at (7.25,11.5) {\textbf{Step 2}};
				\node [font=\normalsize] at (4,11.5) {\textbf{Step 3}};
				\draw  (2.25,13.5) -- (14.5,13.5) -- (15,11.25) -- (2.75,11.25) -- cycle;
			\end{circuitikz}
		}%
		\caption{A flowchart to represent Algorithm~\ref{def:alg} and each approximation step of \reff{md:ecDRO}}
		\label{fig:flowchart}
	\end{figure}

	We would like to make the remarks for Algorithm~\ref{def:alg}.
	\bit
	
	\item In Step 0, the set $\cone{Y}$ is often closed and given
	as a cone given by
	linear, second order or semidefinite constraints.
	We refer to \cite[Chapter~1]{NieBook}
	for how this can be done in computation.

	\item In Step~1, optimization problems \reff{eq:mom_sos}--\reff{eq:polydual}
	can be efficiently solved with the software
	\texttt{GloptiPoly3} \cite{GloptiPoly3},
	\texttt{YALMIP} \cite{LofbergYalmip}
	and \texttt{SeDuMi} \cite{JSturmSedumi}.
	Examples for using these software
	can be found in the book \cite{NieBook}.
	For a given relaxation order $k$,
	the dimension of variable $w$ in \reff{eq:mom_sos} is $\binom{n+2t}{2t}$,
	and the dimension of the variable triple $(\gamma,y,z)$ in \reff{eq:polydual} is
	$1+\binom{p+d}{d}+\binom{p+2k}{2k}$.
	Generally, one can assume \reff{eq:mom_sos}--\reff{eq:polydual} have optimizers.
	This can be guaranteed by strictly feasible points. For instance,
	this is the case if $\qmod{g}$ is archimedean, the set $X$
	has an interior point $\hat{x}$ and $h(\hat{x}, \xi)>0$ for all $\xi \in S$.
	We refer to \cite[Section~8.3]{NieBook}
	for more detailed theory about this.

	\item
	In Steps~2 and 3, the semidefinite program \reff{eq:TKMP}
	and the rank condition \reff{eq:flat}
	are used to check the membership $y^*\in\mathscr{R}_d(S)$.
	In particular, \reff{eq:flat} is called {\it flat truncation}
	(see \cite{JNieFlatTruncation,NieBook}).
	If \reff{eq:flat} holds, then $\hat{z}^*$ has a flat truncation
	and $y^*$ must admit a measure supported in $S$,
	i.e., $y^*\in\mathscr{R}_d(S)$.
	This method works very efficiently in computational practice.
	Under some genericity assumptions, the convergence of the loop
	of Steps~2 and 3 is shown in \cite{JNieAtruncated}.
	We refer to \cite{JNieAtruncated} and
	\cite[Section~8.2]{NieBook} for more details.

	\item
	In Step~3, if \reff{eq:flat} is satisfied, then there exist
	$r \coloneqq \rank M_{d_1}[\hat{z}^*]$
	distinct points $u_1,\dots, u_r\in S$ such that
	\[ y^* = \int[\xi]_d{\tt d}\mu\quad \mbox{with}\quad
	\mu = \theta_1\delta_{u_1}+\cdots+\theta_r\delta_{u_r}, \]
	where each $\theta_i>0$ and $\delta_{u_i}$
	stands for the unique Dirac measure supported at $u_i$.
	We refer to \cite[Section~2.7]{NieBook} for
	how to determine such measure $\mu$.

	\eit

	In Step~4, $x^* = \pi(w^*)$ is a global optimizer of \reff{md:ecDRO}
	if it is SOS-convex and \reff{eq:Kdual} holds.
	Interestingly, this is also true under some other conditions.
	More details are in Theorems~\ref{thm:SOSconvex} and \ref{thm:rank1}.
	The following theorem justifies Step~2.

	\begin{theorem}\label{thm:tightg}
		Assume \reff{eq:mom_sos}--\reff{eq:polydual} have no duality gap.
		Suppose $w^*$ is the optimizer of \reff{eq:mom_sos} and $(\gamma^*, y^*, z^*)$
		is the optimizer of \reff{eq:polydual} at some relaxation order $k$.
		If $y^*\in\mathscr{R}_d(S)$, then $w^*$ is also an optimizer of \reff{eq:momrel1}.
	\end{theorem}
	\begin{proof}
		Note $f^*$ is the optimal value of \reff{eq:momrel1}. Let $f_k$
		be the optimal value of \reff{eq:mom_sos} for the relaxation order $k$.
		Since \reff{eq:mom_sos} is a restriction of \reff{eq:momrel1} for every $k$, we have
		\[
		f^*\,\le\, f_k \,=\, \langle f, w^*\rangle.
		\]
		Note \reff{eq:polydual} is a relaxation for the dual problem of \reff{eq:momrel1}.
		Write the dual problem
		\[
		\left\{
		\begin{array}{cl}
			\max\limits_{(\gamma,y)} & \gamma\\
			\st & f(x) - y^TH\cdot [x]_{2t} - \gamma\in\qmod{c}_{2t},\\
			& y\in cone(Y)\cap \mathscr{R}_d(S),\\
			& \gamma\in \re.
		\end{array}
		\right.
		\]
		If $y^*\in \mathscr{R}_d(S)$,
		then $(\gamma^*, y^*)$ is the optimizer of the above problem.
		So we have $\gamma^*\le f^*$ by the weak duality between \reff{eq:momrel1} and its dual.
		Since \reff{eq:mom_sos}-\reff{eq:polydual} have no duality gap,
		\[
		f_k \,=\, \gamma^*\,\le \, f^*.
		\]
		This implies $f_k = f^*$, so $w^*$ is also an optimizer of \reff{eq:momrel1}.
	\end{proof}

	Next, we discuss when Algorithm~\ref{def:alg}
	returns a global optimizer for the DRO~\reff{md:ecDRO}.

	\subsection{The SOS-convex case}
	\label{ssc:cvx}
	
	We consider the case that
	\reff{md:ecDRO} is an SOS-convex optimization problem.

	\begin{theorem}\label{thm:SOSconvex}
		Assume \reff{eq:Kdual} holds and \reff{eq:mom_sos}--\reff{eq:polydual}
		have no duality gap.
		Suppose $f, -c_1,\ldots,-c_m$ are all SOS-convex and
		$\mathbb{E}_{\mu}[h(x,\xi)]$ is SOS-concave for every $\mu\in\mc{M}$.
		If Algorithm~\ref{def:alg} terminates at some loop of order $k$,
		then the output $x^*$ is a global optimizer of \reff{md:ecDRO}.
	\end{theorem}
	\begin{proof}
		Suppose $w^*$ is the optimizer of \reff{eq:mom_sos}
		and Algorithm~\ref{def:alg} terminates at the loop of order $k$.
		By Theorem~\ref{thm:sosproj} and \ref{thm:tightg},
		$w^*$ is an optimizer of \reff{eq:momrel1} and
		$x^* = \pi(w^*)$ is an optimizer of \reff{eq:momrel}.
		Since \reff{eq:Kdual} holds and $\mathbb{E}_{\mu}[h(x,\xi)]$ is SOS-concave
		for all $\mu\in \mc{M}$, the optimization\reff{eq:momrel}
		is equivalent to \reff{md:ecDRO}, by Theorem~\ref{thm:eqv1}.
		So $x^*$ is also a global optimizer of \reff{md:ecDRO}.
	\end{proof}

	\subsection{The non-SOS-convex case}
	\label{ssc:notcvx}
	
	When the DRO \reff{md:ecDRO} is not SOS-convex,
	Algorithm~\ref{def:alg} may still find a global optimizer of \reff{md:ecDRO},
	under some other assumptions.

	\begin{theorem}\label{thm:rank1}
		Assume \reff{eq:Kdual} holds and \reff{eq:mom_sos}-\reff{eq:polydual}
		have no duality gap. Suppose Algorithm~\ref{def:alg} outputs $x^*=\pi(w^*)$
		at a loop of relaxation order $k$
		where $w^*$ is the optimizer of \reff{eq:mom_sos}.
		If \rank\,$M_t[w^*] = 1$,
		then $x^*$ is a global optimizer of \reff{md:ecDRO}.
	\end{theorem}
	\begin{proof}
		By Theorem~\ref{thm:tightg}, $w^*$ is a global optimizer of \reff{eq:momrel1}.
		If $\rank M_{t}[w^*] = 1$,  then for $x^* = \pi(w^*)$,
		we must have
		\[
		M_{t}[w^*] = [x^*]_{t}[x^*]_{t}^T.
		\]
		See \cite[Section~5.3]{NieBook} for this fact.
		So,  $f(x^*) = \langle f,w^*\rangle$ and $x^*$ is feasible for \reff{md:ecDRO},
		since $c_i(x^*) = \langle c_i, w^*\rangle \ge 0$ for each $i$ and
		\[
		\mathbb{E}_{\mu}[h(x^*,\xi)] \,=\,
		\mathbb{E}_{\mu}[(Hw^*)^T[\xi]_d]\,\ge\, 0,
		\quad \forall \mu\in\mc{M}.
		\]
		Since \reff{eq:momrel1} is a relaxation of \reff{md:ecDRO},
		we know $x^*$ must also be a global optimizer of \reff{md:ecDRO}.
	\end{proof}

	Suppose $(\gamma^*, y^*, z^*)$ is the optimizer of \reff{eq:polydual}
	for some relaxation order $k$.
	In Algorithm~\ref{def:alg}, we can get the membership
	$y^*\in\mathscr{R}_d(S)$ if
	\[
	\exists d_1\in [d_0,k] \quad \mbox{such that}\quad
	\rank\,M_{d_1}[z^*] \,= \, 1.
	\]
	This is a special case for the flat truncation condition \reff{eq:flat}.
	If it is met, then $y^*$ must admit a Dirac measure supported at
	the point $\xi^* = (y_{e_1}^*,\ldots, y_{e_p}^*)$.
	Therefore, we get the following corollary.
	
	\begin{cor}\label{coro:rank1}
		Assume \reff{eq:Kdual} holds and \reff{eq:mom_sos}--\reff{eq:polydual}
		have no duality gap.
		Suppose $w^*$ is an optimizer of \reff{eq:mom_sos} and
		$(\gamma^*, y^*,z^*)$ is an optimizer of \reff{eq:polydual}
		for some relaxation order $k$.
		If $\rank\,M_t[w^*] = 1$ and $\rank\, M_{d_1}[z^*] = 1$
		for some degree $d_1\in [d_0,k]$,
		then $\gamma^*$ is the optimal value and $x^* = \pi(w^*)$
		is the optimizer of \reff{md:ecDRO}.
	\end{cor}
	\begin{proof}
		The conclusion follows from Theorem~\ref{thm:rank1}
		and the above discussion.
	\end{proof}
	
	We conclude this section with an example.
	
	\begin{example}\label{ex:bi_xi_nolin}
		Consider the DRO problem
		\begin{equation}\label{eq:bi_xi_nolin}
			\left\{
			\begin{array}{cl}
				\min\limits_{x\in\re^2} &  x_1^2 + 2x_1x_2 + x_2\\
				\st & \inf\limits_{\mu\in\mc{M}}\,\mathbb{E}_{\mu}
				[x_1x_2-x_1\xi_1^2-x_2^2\xi_2^2]\ge 0,\\
				& 1-x_1^2-x_2^2\ge 0,
			\end{array}
			\right.
		\end{equation}
		where $\xi = (\xi_1,\xi_2)$, $S = [0,1]^2$ and,
		\[
		Y = \big\{ y\in\re^{\N_2^2}
		: y_{00} = 1, y_{10}+y_{20}\le 1, y_{02}+y_{02}\le 2
		\big\}.
		\]
		The objective function of \reff{eq:bi_xi_nolin} is not convex.
		For $t = 1$ and $k=1$, we solve \reff{eq:mom_sos}--\reff{eq:polydual}
		and get the optimizers
		\[
		w^* = \big(1, -\frac{1}{6}, -\frac{1}{6}, \frac{1}{36}, \frac{1}{36}, \frac{1}{36}\big),
		\quad  y^* = (1, 0, 1, 0, 0, 1).
		\]
		Both $M_1[w^*]$ and $M_1[y^*]$ are rank one.
		By Corollary~\ref{coro:rank1},
		the DRO problem \reff{eq:bi_xi_nolin} has a global optimizer
		\[
		x^* = \pi(w^*) = \frac{1}{6}(-1, -1),
		\]
		and its global minimum is $ f_{min} = f(x^*) = -1/12$.
	\end{example}

	\subsection{A heuristic trick}
	
	We consider the case that the DRO \reff{md:ecDRO} is not solved as in Theorems~\ref{thm:SOSconvex} and \ref{thm:rank1}.
	Suppose $w^*$ and $(\gamma^*,y^*,z^*)$ are computed from Algorithm~\ref{def:alg}
	for some relaxation order $k$.
	Observe that if $y^*\in\mathscr{R}_d(S)$
	and $w^* = [x^*]_{2t}$, then
	\[ h(x^*)^Ty^* \,=\, (Hw^*)^Ty^* \,=\,
	\inf\limits_{\mu\in\mc{M}} \mathbb{E}_{\mu} \big[ h(x^*)^T[\xi]_d \big] \ge 0.
	\]
	When the rank-one condition in Theorem~\ref{thm:rank1} fails,
	the projection $x^* = \pi(w^*)$ may or may not be an optimizer.
	It may even not be feasible for the original DRO.
	To get a robust candidate solution, a heuristic trick is
	to use $h(x)^Ty^* \ge 0$ to approximate the robust constraint in \reff{md:ecDRO}.
	Therefore, we consider the deterministic polynomial optimization problem
	\begin{equation}\label{eq:heurisPO}
		\left\{\begin{array}{cl}
			\min\limits_{x\in X} & f(x)\\
			\st & h(x)^Ty^*\ge 0.
		\end{array}
		\right.
	\end{equation}
	One can solve it by Moment-SOS relaxations.
		Under some genericity assumptions, the Moment-SOS hierarchy 
        either returns an optimizer of \reff{eq:heurisPO} or detects its infeasibility.
		This is a well-studied topic in polynomial optimization.
		We refer to \cite[Chapter 5]{NieBook} for more detailed introductions.
	Suppose $\hat{x} \in X$ is the optimizer of \reff{eq:heurisPO}.
	It is feasible for \reff{md:ecDRO} if and only if
	the robust constraint is satisfied.
	This can be verified by solving the linear conic optimization
	\begin{equation}\label{eq:feasveri}
		\left\{\begin{array}{cl}
			\min\limits_{(y,z)} & h( \hat{x} )^Ty\\
			\st & y\in cone(Y),\\
			& L_{g_i}^{(k)}[z]\succeq 0,\,i = 1, \ldots, m_1, \\
			& y=z|_d,\, M_k[z]\succeq 0,\\
			& z \in \re^{ \N^n_{2k} }.
		\end{array}
		\right.
	\end{equation}
	for relaxation orders $k\ge d_0$. Let $\eta_k$ denotes the optimal value of \reff{eq:feasveri}.
	Since each $y\in K$ can be extended to a feasible pair $(y,z)$ of \reff{eq:feasveri},
	\[
	\eta_k \,\le\, \inf_{y\in K} \, h( \hat{x} )^Ty \,=\,
	\inf_{\mu\in\mc{M}}\mathbb{E}_{\mu}[h( \hat{x},\xi)].
	\]
	If $\eta_k\ge 0$ for some $k\ge d_0$, then $\hat{x}$ is feasible for the DRO \reff{md:ecDRO}.
	It is a global optimizer if we, in addition, have $f(\hat{x}) = f^*$.
	This is because we must have $f(\hat{x}) = f_{min}$, since $f^*$
	is the optimal value of \reff{eq:momrel1}, which is a lower bound of
	the minimum value $f_{min}$ of \reff{md:ecDRO}.
	If $\eta_k<0$ is negative for all $k$,
	we do not know how to solve the DRO \reff{md:ecDRO}.
	This question is mostly open, to the best of the authors's knowledge.

	\begin{example}
		Consider the DRO problem
		\[
		\left\{
		\begin{array}{cl}
			\min\limits_{x\in\re^2} &  x_1x_2 - x_2^2\\
			\st & \inf\limits_{\mu\in\mc{M}}\,\mathbb{E}_{\mu}
			[x_1x_2+x_1\xi_1^2-x_2^2\xi_2^2]\ge 0,\\
			& 1-x_1^2+x_2^2\ge 0,
		\end{array}
		\right.
		\]
		where $\xi = (\xi_1,\xi_2)$, $S = [0,1]^2$ and $Y$ is as in Example~\ref{ex:bi_xi_nolin}.
		The objective and feasible set of this DRO are both nonconvex. For $t = 1$ and $k=1$,
		we solve \reff{eq:mom_sos}--\reff{eq:polydual}
		and get the optimizers:
		\[\begin{array}{r}
			w^* = (1.0000, 0.7167, 0.4523, 1.4200, 0.9370, 0.9370),\\
			y^* = (1.0000, 0.0000, 1.0000, 0.0000, 0.0000, 1.0000) .
		\end{array}\]
		Since $\rank\, M_1[y^*] = 1$, we know $y^*\in\mathscr{R}_2(S)$ and hence
		$f^* = \langle f, w^*\rangle = 0.0000$ by Theorem~\ref{thm:tightg}.
		However, $\rank\, M_1[w^*]>1$, so Theorem~\ref{thm:rank1} does not imply $x^* = \pi(w^*)$ is a global minimizer.
		Then, we solve \reff{eq:heurisPO} and get the optimizer $\hat{x} = (0.0000, 0.0000)$.
		Solving \reff{eq:feasveri}, one can check that $\hat{x}$ is feasible for this DRO.
		Since $f(\hat{x}) = f^*$, we know $\hat{x}$ is a global optimizer and the minimum value $f_{min} = 0.0000$.
	\end{example}

	\section{Numerical Experiments}
	\label{sec:numerical}

	In this section, we present numerical experiments for applying
	Algorithm \ref{def:alg} to solve DRO with nonlinear robust constraints.
	The computation is implemented in MATLAB R2018a,
	in a Laptop with CPU 8th Generation Intel\textregistered Core\texttrademark i5-8250U and RAM 16 GB.
	The software \texttt{GloptiPoly3} (version 3.10) \cite{GloptiPoly3},
	\texttt{YALMIP} (version R20210331) \cite{LofbergYalmip}
	and \texttt{SeDuMi} (version 1.3) \cite{JSturmSedumi} are used for
	implementing Algorithm~\ref{def:alg}.
	For neatness of presentation, we only display four decimal digits
	for computational results.
	
	\begin{example}\label{ex:sosconvex}
		Consider the DRO problem
		\[
		\left\{\begin{array}{cl}
			\min\limits_{x\in\re^2} &2x_1-x_2+(x_1-x_2)^2\\
			\st &  \inf\limits_{\mu\in\mc{M}} \mathbb{E}_{\mu}[x_1\xi_1^2-x_2\xi_2^2-x_1^2\xi_1^3-x_2^2\xi_2^3]\ge 0,\\
			& x_1-x_2\ge 0,\quad 1-x_1^2-x_2^2\ge 0,
		\end{array}
		\right.
		\]
		where $\xi= (\xi_1,\xi_2)$ and $\mc{M}$ is given with
		\[
		S = \{\xi = (\xi_1,\xi_2): \xi_1\ge 0,\, \xi_2-\xi_1\ge 0,\, 1-\xi_1-\xi_2\ge 0\},
		\]
		\[
		Y = \left\{y\in\re^{\N_3^2}\left|\begin{array}{c}
			y_{00} = 1,\,  y_{00}\le 2y_{10}+2y_{01}\\
			y_{10}+y_{01}\le 2y_{20}+2y_{02}\\
			y_{20}+y_{02}\le 2y_{30}+2y_{03}
		\end{array}\right.\right\}.
		\]
		Since $Y$ is given by linear constraints, we can explicitly write
		\[
		cone(Y) = \Big\{y\in\re^{\N_3^2}: 0\le \frac{y_{00}}{8}\le \frac{y_{10}+y_{01}}{4}\le
		\frac{y_{20}+y_{02}}{2}\le y_{30}+y_{03} \Big\}.
		\]
		The objective function is SOS-convex, and each $c_i(x)$ is SOS-concave.
		Compute the Hessian matrix
		\[
		-\nabla_x^2 h(x,\xi) = 2\left[\begin{array}{cc}
			\xi_1^3 &  0\\
			0 & \xi_2^3
		\end{array}\right].
		\]
		Since $S\subseteq \re_+^2$, $-\mathbb{E}_{\mu}[\nabla_x^2 h(x,\xi)]$ is a constant psd matrix,
		thus $\mathbb{E}_{\mu}[h(x,\xi)]$ is SOS-concave for all $\mu\in \mc{M}$.
		It took around $0.58$ second by Algorithm~\ref{def:alg}.
		For $t = 1, k=2$, we solve the dual pair \reff{eq:mom_sos}--\reff{eq:polydual} and get the optimizers
		\[\begin{array}{r}
			w^* = ( 1.0000, -0.2450, -0.3291, 0.0600, 0.0806, 0.1083),\\
			y^* = (9.2405, 3.4833, 3.4833, 1.7416, 1.7416, 1.7416, 0.8708,\\
			0.8708,0.8708, 0.8708).
		\end{array}\]
		For $k_1 = 2$, we solve the trunacted moment problem \reff{eq:TKMP} and the flat truncation
		\reff{eq:flat} is satisfied. It implies $y^*\in\mathscr{R}_3(S)$ and we obtain the decomposition
		\[
		y^* = 2.2740\bbm 0.0000\\0.0000\ebm_3 +
		6.9665\bbm 0.5000\\0.5000\ebm_3.
		\]
		Clearly, \reff{eq:Kdual} holds. By Theorem~\ref{thm:SOSconvex},
		this DRO problem has a global optimizer
		\[
		x^* = \pi(w^*)= (-0.2450, -0.3291),
		\]
		and its global minimum is $ f_{min} = f(x^*)  = -0.1537$.
	\end{example}
	
	\begin{example}\label{ex:sim_nonSOS}
		Consider the DRO problem
		\[
		\left\{\begin{array}{cl}
			\min\limits_{x\in\re^3} & x_1^3 + (x_2-x_1-x_3)^2 + x_3^3\\
			\st & \inf\limits_{\mu\in\mc{M}} \mathbb{E}_{\mu}[h(x,\xi)]\ge 0,\\
			& 1\le \|x\|^2 \le 4,\, x_3-x_1-x_2\ge 0,
		\end{array}
		\right.
		\]
		where $\xi = (\xi_1,\xi_2)$, $S = [-1,1]^2$ and
		\[
		h(x,\xi) = x_3\xi_1^4 +x_1x_3\xi_2^4 + (x_2-x_1-1)\xi_1^2\xi_2^2,
		\]
		\[
		Y = \Big\{
		y\in\re^{\N_4^2}:
		y_{00} = 1,\, \, y_{30}\ge 2y_{03},\,\sum\limits_{|\alpha|\ge 1} y_{\alpha}^2= 5
		\Big\}.\]
		The objective function is not convex,
		and $h(x,\xi)$ is not concave in $x$.
		In addition, the moment set $Y$ is not convex.
		Since $y_{00} = 1$ and $\sum_{|\alpha|\ge 1}y_{\alpha}^2 = 5$, we have
		\[ \|y\|_2^2 = y_{00}^2+\sum_{|\alpha|\ge1 , \alpha\in\N_4^2} y_{\alpha}^2 = 1+5 = 6. \]
		The convex hull of $\{y: \|y\|_2 = \sqrt{6}\}$ is determined by $\|y\|_2\le \sqrt{6}$.
		Then we have an explicit expression of the conic hull
		\[
		cone(Y) = \big\{
		y\in\re^{\N_4^2}:
		y_{30}\ge 2y_{03},\,
		\|y\|_2\le \sqrt{6}\, y_{00}
		\big\}.
		\]
		It took around $1.11$ second by Algorithm~\ref{def:alg}.
		For $t = 2$ and $k = 2$, we solve \reff{eq:mom_sos}--\reff{eq:polydual}
		and get the optimizers
		\[
		\begin{array}{r}
			w^* = (1.0000, -1.9078, -0.6004, 0.0000, 3.6395, 1.1454, 0.0000, 0.3605, 0.0000,\\
			0.0000, -6.9433, -2.1852, 0.0000, -0.6877, 0.0000, 0.0000, -0.2164,\\
			0.0000, 0.0000, 0.0000, 13.2460, 4.1688, 0.0000, 1.3120, 0.0000, 0.0000,\\
			0.4129, 0.0000, 0.0000, 0.0000, 0.1300, 0.0000, 0.0000, 0.0000, 0.0000),\\
			y^* = (9.6684, 0.4256, -1.1012, 2.6152, 0.0000, 3.1062, 0.3944, 0.0001, 0.0000,\\
			-1.0308, 1.8761, 0.0000, 0.0000, 0.0000, 2.3538).
		\end{array}
		\]
		For $k_1 = 3$, we solve the truncated moment problem \reff{eq:TKMP} and the flat truncation
		\reff{eq:flat} is satisfied. It implies $y^*\in\mathscr{R}_4(S)$ and we obtain the decomposition
		\[
		y^* = 1.8437\left[\begin{array}{r}-0.8693\\0.0383\end{array}\right]_4 +
		2.9456\left[\begin{array}{r} 0.0308\\-0.9100\end{array}\right]_4 +
		\]
		\[
		\quad 2.1548\left[\begin{array}{r}-0.0274\\0.7513\end{array}\right]_4 +
		2.3413\left[\begin{array}{r}0.8578\\-0.0092\end{array}\right]_4.
		\]
		The $Y$ and $S$ are both compact, so \reff{eq:Kdual} holds.
		Since $\rank\, M_2[w^*] = 1$, by Theorem~\ref{thm:rank1}, this DRO has a global optimizer
		\[
		x^*=\pi(w^*) = (-1.9078, -0.6004, 0.0000),
		\]
		and its global minimum is $f_{min} = f(x^*) = -5.2341$.
	\end{example}
	
	\begin{example}\label{ex:com:nonSOS}
		Consider the DRO problem
		\[
		\left\{
		\begin{array}{cl}
			\min\limits_{x\in\re^4} & x_1(x_2-x_4)+x_2(x_1+x_3)\\
			\st & \inf\limits_{\mu\in\mc{M}}\mathbb{E}_{\mu}[h(x,\xi)]\ge 0,\\
			& 1-\|x\|^2 \ge 0,\, x\ge 0,\\
			& x_3+x_4 - x_1^4-x_2^4 \ge 0,
		\end{array}
		\right.
		\]
		where $\xi = (\xi_1,\xi_2)$ and $h(x,\xi)$ is given as
		\[
		h(x,\xi) = x_3(\xi_1^4+\xi_2^4) - (x_4+x_1x_4)\xi_1^2\xi_2^2
		+ x_1x_2\xi_1^2 + x_1^2\xi_2^2 - x_2x_4\xi_1\xi_2.
		\]
		In the above, assume $\mc{M}$ is given by
		\[
		S = \{\xi\in\re^2: 1-\|\xi\|^2\ge 0\},
		\]
		\[
		Y = \left\{
		y\in\re^{\N_4^2}\left|\begin{array}{c}
			y_{00} = 1,
			\bbm y_{20} & y_{11}\\y_{11} & y_{02}\ebm\preceq \frac{1}{2}I_2,\,
			\bbm y_{40} & y_{31} & y_{22}\\
			y_{31} & y_{22} & y_{13}\\
			y_{22} & y_{13} & y_{04}\ebm\preceq \frac{1}{4}I_3
		\end{array}
		\right.
		\right\}.\]
		The $cone(Y)$ can be explicitly expressed as
		\[
		cone(Y) = \left\{
		y\in\re^{\N_4^2}\left|\begin{array}{c}
			y_{00}\ge 0,
			\bbm y_{20} & y_{11}\\y_{11} & y_{02}\ebm\preceq \frac{y_{00}}{2}I_2,\,
			\bbm y_{40} & y_{31} & y_{22}\\
			y_{31} & y_{22} & y_{13}\\
			y_{22} & y_{13} & y_{04}\ebm\preceq \frac{y_{00}}{4}I_3
		\end{array}
		\right.
		\right\}.
		\]
		The objective function is not convex, and $h(x,\xi)$ is not concave in $x$.
		It took around $2.68$ seconds by Algorithm~\ref{def:alg}.
		For $t = 2$ and $k = 2$, we solve \reff{eq:mom_sos}--\reff{eq:polydual}
		and get the optimizers
		\[
		\begin{array}{r}
			w^* = (1.0000, 0.7391, 0.0000, 0.1333, 0.6602, 0.5463, 0.0000, 0.0985, 0.4880,\\
			0.0000, 0.0000, 0.0000, 0.0178, 0.0880, 0.4359, 0.4038, 0.0000, 0.0728,\\
			0.3607, 0.0000, 0.0000, 0.0000, 0.0131, 0.0650, 0.3222, 0.0000, 0.0000,\\
			0.0000, 0.0000, 0.0000, 0.0000, 0.0024, 0.0117, 0.0581, 0.2878, 0.2984,\\
			0.0000, 0.0538, 0.2666, 0.0000, 0.0000, 0.0000, 0.0097, 0.0481, 0.2381,\\
			0.0000, 0.0000, 0.0000, 0.0000, 0.0000, 0.0000, 0.0018, 0.0087, 0.0429,\\
			0.2127, 0.0000, 0.0000, 0.0000, 0.0000, 0.0000, 0.0000, 0.0000, 0.0000,\\
			0.0000, 0.0000, 0.0003, 0.0016, 0.0077, 0.0384, 0.1900),\\
			y^* = (2.3070, 0.0000, 0.0000, 0.1594, 0.0049, 0.0705, 0.0000, 0.0000, 0.0000,\\
			0.0000, 0.1104, 0.0034, 0.0489, 0.0015, 0.0217).
		\end{array}
		\]
		For $k_1 = 3$, we solve the truncated moment problem \reff{eq:TKMP} and the flat truncation \reff{eq:flat}
		is satisfied. It implies $y^*\in\mathscr{R}_4(S)$ and we obtain the decomposition
		\[
		y^* = 0.0602\left[\begin{array}{r}-0.8324\\-0.5540\end{array}\right]_4 +
		0.0548\left[\begin{array}{r}-0.8324\\0.5540\end{array}\right]_4 +
		2.0770\bbm 0.0000\\0.0000\ebm_4
		\]
		\[\quad
		+0.0549\left[\begin{array}{r}0.8324\\-0.5539\end{array}\right]_4 +
		0.0601\left[\begin{array}{r}0.8325\\0.5539\end{array}\right]_4.
		\]
		The \reff{eq:Kdual} is clearly satisfied.
		Since $\rank\, M_2[w^*] = 1$, by Theorem~\ref{thm:rank1}, this DRO has a global optimizer
		\[
		x^* = \pi(w^*) = ( 0.7391, 0.0000, 0.1333, 0.6602),
		\]
		and its global minimum is $f_{min} = f(x^*) = -0.4880$.
	\end{example}

	Then we give two numerical examples with applications in portfolio optimization.
	In the first example, we compare the performances between linear portfolio selection model and the mean-variance model.
	In the second example, we explore the impact of higher-order moments
	in the mean-variance model.
	In both examples, we construct the ambiguity set using the method proposed in \cite{NieDro2023},
	where the moment constraining set
	\begin{equation}\label{eq:numY}
		Y\, =\, \{y\in\re^{\N_d^n}: l\le y\le u\}
	\end{equation}
	for scalar vectors $l,u$ computed from random samples of $\xi$.
	
	\begin{example}
		Consider a portfolio comprising $n$ assets,
		of which the random returns are denoted by $\xi = (\xi_1,\ldots, \xi_n)^T$.
		Let $x = (x_1,\ldots,x_n)$ denote the vector of investing weights;
		i.e.,
		\[
		x\in \, \Delta \,\coloneqq\, \{x\in\re^n: x\ge 0,\,e^Tx = 1\}.
		\]
		In portfolio optimization, the goal is to seek optimal investment allocations tailored to
		different risk appetites. In this example, we study the linear portfolio selection model \cite{NieDro2023} and
		the mean-variance (M-V) model \cite{YYZportfolio}, and compare the numerical results
		based on the same sampling set.
		
		Suppose $n = 3$, $\xi = (\xi_1,\xi_2,\xi_3)$, $S = [0,1]^3$
		and each $\xi_i$ is independently distributed.
		Assume $\xi_1$ follows the uniform distribution,
		$\xi_2$ follows the truncated standard normal distribution,
		and $\xi_3$ follows the truncated exponential distribution with mean value $0.5$.
		We use {\tt MATLAB} to randomly generate $150$ samples of $\xi$. The sample average is
		\begin{equation}\label{eq:samavg}
			\nu\, = \, (0.5132, 0.4598, 0.4356)^T.
		\end{equation}
		Using the method introduced in \cite{NieDro2023},
		we construct $Y = \{y\in\re^{\N_2^3}: l\le y \le u\}$ with
		\[\begin{array}{r}
			l = (1.0000, 0.4849, 0.3942, 0.3880, 0.3258, 0.1922, 0.1970, 0.2164, 0.1640, 0.2190),\\
			u = (1.0000, 0.5414, 0.5254, 0.4833, 0.3679, 0.2544, 0.2422, 0.3674, 0.2271, 0.3216).
		\end{array}\]
		The moment ambiguity set is
		\begin{equation}\label{eq:portfolioM}
			\mc{M} = \{\mu: \supp{\mu}\subseteq [0,1]^3, \mathbb{E}_{\mu}([\xi]_2)\in Y\}.
		\end{equation}
		Since  $l_1 = u_1 = 1.0000$ (here $l_1,u_1$ denotes the first entry of $l,u$ respectively),
		all measures in $\mc{M}$ are probability measures.

		\noindent
		(i) First, we study the linear portfolio selection model
		\begin{equation}\label{eq:linear_portfolio}
			\min\limits_{x\in \Delta} \max\limits_{\mu\in\mc{M}}\,
			\mathbb{E}_{\mu}[-x^T\xi].
		\end{equation}
		It can be equivalently reformulated as
		\begin{equation}\label{eq:ref_lin_port}
			\left\{\begin{array}{cl}
				\min\limits_{(x_0,\bar{x})} & x_0\\
				\st & \inf\limits_{\mu\in\mc{M}}\mathbb{E}_{\mu}[x_0+x^T\xi]\ge 0,\\
				& x = (\bar{x}, 1-e^T\bar{x})\in \Delta,\, x_0\in\re, \, \bar{x}\in\re^2,
			\end{array}\right.
		\end{equation}
		where $e = (1,1)^T$. The \reff{eq:ref_lin_port} is SOS-convex with affine objective function and constraints.
		It took us around $1.03$ second by Algorithm~\ref{def:alg}.
		For $t=1$ and $k=1$, we solve the dual pair \reff{eq:mom_sos}--\reff{eq:polydual} and get the optimizers
		\[\begin{array}{r}
			w^* = ( 1.0000, -0.4849, 1.0000, 0.0000, 0.5071, -0.4668, 0.0000, 1.1392, 0.0083, 0.3393),\\
			y^* = ( 1.0000, 0.4849, 0.4372, 0.4214, 0.3477, 0.2238, 0.2197, 0.2995, 0.1956, 0.2754).
		\end{array}\]
		For $k_1 = 2$, we solve the truncated moment problem \reff{eq:TKMP} and the flat truncation \reff{eq:flat}
		is satisfied. It implies $y^*\in\mathscr{R}_2(S)$ and we obtain the decomposition
		\[
		y^* = 0.1590\bbm0.0000\\0.8169\\0.1077\ebm_2 +
		0.1641\bbm 0.0000\\0.0000\\0.6200\ebm_2 +
		0.0845\bbm0.6603\\0.0000\\ 0.0000 \ebm_2 +
		\]
		\[
		+ 0.1556\bbm 0.7213\\0.3366\\0.0000\ebm_2
		+ 0.0670\bbm 0.7023\\0.0000\\0.5843\ebm_2
		+ 0.3698\bbm 0.7298\\0.6893\\0.7123
		\ebm_2.
		\]
		Clearly, \reff{eq:Kdual} holds. By Theorem~\ref{thm:SOSconvex}, the DRO \reff{eq:ref_lin_port}
		has a global optimizer
		\[
		(x_0^*,\bar{x}^*) = \pi(w^*) =  (-0.4849,\, 1.0000,\, 0.0000).
		\]
		Hence \reff{eq:linear_portfolio} has a global optimizer
		\[
		x^* = (\bar{x}^*, 1-e^T\bar{x}^*) = (1.0000,\, 0.0000,\, 0.0000).
		\]
		This is a high risk investment strategy,
		as it involves allocating all resources into a single asset.

		\noindent
		(ii) Then, we consider the M-V model
		\begin{equation}\label{eq:M-V}
			\min\limits_{x\in\Delta}\max\limits_{\mu\in \mc{M}}\,
			\mathbb{E}_{\mu}\big[-x^T\nu
			+ (x^T\xi-x^T\nu)^2\big].
		\end{equation}
		It can be equivalently reformulated as
		\begin{equation}\label{eq:M-V:dro}
			\left\{\begin{array}{cl}
				\min\limits_{(x_0,\bar{x})} & x_0\\
				\st & \inf\limits_{\mu\in\mc{M}} \mathbb{E}_{\mu}\big[x_0+x^T\nu
				- (x^T\xi-x^T\nu)^2\big] \ge 0,\\
				& x = (\bar{x}, 1-e^T\bar{x}) \in\Delta,\,x_0\in\re,\,\bar{x}\in\re^2,
			\end{array}
			\right.
		\end{equation}
		where $e = (1,1)^T$.
		Clearly, this is an SOS-convex DRO.
		It took around $1.50$ second by Algorithm~\ref{def:alg}.
		For $t = 1$ and $k = 1$, we solve the dual pair \reff{eq:mom_sos}--\reff{eq:polydual} and get the optimizers
		\[\begin{array}{r}
			w^* = (1.0000, -0.3907,  0.7277,  0.1326,  0.4039, -0.2843, -0.0518,\\
			0.5296,  0.0965,  0.0176),\\
			y^* =  (1.0000, 0.4849, 0.3942, 0.3880, 0.3679,  0.2544,  0.2422,\\  0.3674,  0.2271,  0.3216).
		\end{array}\]
		For $k_1 = 2$, we solve the truncated moment problem \reff{eq:TKMP} and the flat truncation \reff{eq:flat}
		is satisfied. It implies $y^*\in\mathscr{R}_2(S)$ and we obtain the decomposition
		\[
		y^* = 0.2619\bbm 0.0000\\0.0000\\0.0000\ebm_2 +
		0.0238\bbm 0.0000\\0.0000\\0.8199\ebm_2 +
		0.1447\bbm0.5586\\0.0000\\0.8091 \ebm_2 +
		0.1457\bbm 0.8831\\0.0000\\0.0000\ebm_2
		\]
		\[
		+ 0.0555\bbm 0.0000\\1.0000\\0.0000\ebm_2
		+ 0.2997\bbm 0.7034\\0.9035\\0.8388\ebm_2
		+ 0.0686\bbm 0.9419\\0.9903\\0.0000\ebm_2.
		\]
		Clearly, \reff{eq:Kdual} holds.
		By Theorem~\ref{thm:SOSconvex}, the DRO \reff{eq:M-V:dro} has a global optimizer
		\[
		(x_0^*,\bar{x}^*) = \pi(w^*) =  (-0.3907,\, 0.7277,\, 0.1326).
		\]
		Hence \reff{eq:M-V} has a global optimizer
		\[
		x^* = (\bar{x}^*, 1-e^T\bar{x}^*) = (0.7277,\, 0.1326,\, 0.1397).
		\]
		Compared to the optimizer from the linear model,
		this is a more robust investment strategy.
	\end{example}

	Then we explore the impact of higher-order moments in the M-V model.
	\begin{example}\label{ex:portfolio_large}
		Consider the M-V model as in \reff{eq:M-V} with $n = 10$.
		Assume each $\xi_i$ is independently distributed and follows the truncated normal distribution
		\[ \xi_i \sim\mc{N}_T(0.05 i, 0.03 i, -1,1),\quad i = 1,\ldots,10. \]
		Here $\mc{N}_T(\mu,\sigma, a,b)$ stands for the normal distribution with the mean $\mu$,
		variance $\sigma$ and is truncated within $[a,b]$. For each pair $(M,d)$ such that
		\[
		M\in \{40, 400, 4000\},\quad d\in\{1,2,3\},
		\]
		we made $10$ independent numerical simulations.
		In each simulation, we randomly generate $M$ samples of $\xi$, and use $75\%$ of samples
		to construct moment ambiguity set (using the method proposed in \cite{NieDro2023}) of the form
		\[ \mc{M} = \{\mu: \supp{\mu}\subseteq [-1,1]^{10},\, \mathbb{E}_{\mu}([\xi]_d)\in Y\}. \]
		The rest $25\%$ of samples are used to evaluate out-of-sample performance.
		We ran Algorithm~\ref{def:alg} at the initial relaxation order.
		Since \reff{eq:M-V:dro} is SOS-convex, the computed $x^* = \pi(w^*)$ is always feasible for the DRO.
		For a sample set $\mc{S} = \{\xi^{(1)},\ldots, \xi^{(M_1)}\}$ of $\xi$,
		we evaluate the performance of a candidate solution $x$ over $\mc{S}$ by
		\[
		J(x) = \frac{1}{M_1}\sum\limits_{i=1}^{M_1}\big[-x^T\hat{\nu} + (x^T\xi^{(i)}-x^T\hat{\nu})^2\big],
		\quad\mbox{where}\quad \hat{\nu} = \frac{1}{M_1}\sum\limits_{i=1}^{M_1} \xi^{(i)}.
		\]
		The numerical results are reported in Table~\ref{tab:portfolio_large}.
		The``avg. $J_{\tt in}$'' is used to denote ``average in-sample performance'',
		and ``avg. $J_{\tt out}$'' is used to denote ``average out-of-sample performance''.
		The ``avg. time'' stands for ``average CPU time'', which is counted by seconds.
		\begin{table}[htb!]
			\centering
			\caption{Numerical results for Example~\ref{ex:portfolio_large}}
			\label{tab:portfolio_large}
			\begin{tabular}{c| cc| cc| cc| c}
				\toprule
				& \multicolumn{2}{c}{$M = 40$} & \multicolumn{2}{c}{$M = 400$} & \multicolumn{2}{c|}{$M = 4000$} &
				avg. time\\
				\cmidrule(lr){2-3} \cmidrule(lr){4-5} \cmidrule(lr){6-7}
				& avg. $J_{\tt in}$ & avg. $J_{\tt out}$ & avg. $J_{\tt in}$ & avg. $J_{\tt out}$
				& avg. $J_{\tt in}$ & avg. $J_{\tt out}$ & (sec.)\\
				\midrule
				$d=1$ & $1.7873$ & $1.7647$ & $1.7963$ & $1.8126$ & $1.7933$ & $1.7997$ & $0.45$\\
				$d=2$ & $1.5330$ & $1.5220$ & $1.5381$ & $1.5357$ & $1.5325$ & $1.5348$ & $1.02$\\
				$d=3$ & $1.5387$ & $1.5159$ & $1.5164$ & $1.5191$ & $1.5080$ & $1.5098$ & $7.45$\\
				\bottomrule
			\end{tabular}
		\end{table}
		Since avg. $J_{\tt in}$ (resp. avg. $J_{\tt out}$) is an upper bound for
		the optimal value of this M-V model, smaller avg. $J_{\tt in}$  (resp. avg. $J_{\tt out}$) implies better performance.
		From the data in Table~\ref{tab:portfolio_large}, when $d$ is fixed,
		both in-sample and out-of-sample performances are slightly improved with the increase of the
		sample size $M$.
		On the other hand, for a fixed $M$,
		in-sample and out-of-sample performances are efficiently improved with the increase of moment order $d$.
		In particular, we obtained much better in-sample and out-of-sample performances with a small sample size $M=40$ for $d=2$,
		compared to these based on a much bigger sample size $M=4000$ for $d=1$.
		It implies that hihger-order moments has a big impact on the out-of-sample performance and sample efficiency for this M-V model.
	\end{example}

	\section{Conclusions and discussions}
	\label{sec:con}
	
	This paper studies polynomial DRO problems with polynomial robust constraints.
	By introducing new moment variables, we relax the DRO
	to linear conic optimization problems.
	A Moment-SOS algorithm is proposed to solve the relaxed optimization.
	When the original DRO is SOS-convex,
	we show that the relaxed optimization
	can give a global optimizer of the DRO.
	For nonconvex cases, the relaxed optimization
	may also returns a global optimizer for the DRO.
	Some numerical experiments are also given
	to show the efficiency of the proposed method.
	
	When the polynomials $-\mathbb{E}_{\mu}[h(x,\xi)]$, $f(x)$ and $-c_i(x)$ are convex but not SOS-convex,
		solving the DRO problem is still difficult, even in the absence of uncertainty.
		The difficulty arises from the semidefinite representability.
		The reason is that if $X = \{x\in\re^n: c_i(x)\ge 0,i=1,\ldots, m\}$ 
        is given by concave but not SOS-concave polynomials $c_i(x)$,
		it may be very difficult to represent $X$ by semidefinite programming.
		We refer to \cite[Chapter 7]{NieBook} for this topic.
		For the case that all $c_i(x)$ are concave,
		the NP-hardness of finding a feasible point in $X$ is still unknown, to the best of the authors' knowledge.

	Solving nonconvex DRO problems like \reff{md:ecDRO}
	is usually highly challenging,
	especially when the robust constraint in nonconvex in
	the decision variable $x$.
	Our first relaxation problem~\reff{eq:momrel}
	has linear constraints about $x$,
	while its objective $f(x)$ is still a polynomial function.
	We refer to \cite{NieDro2023} for how to solve
	this kind of optimization problems.
	When the robust constraint in \reff{md:ecDRO}
	are nonlinear nonconvex in $x$,
	it is still mostly an open question for how to relax
	\reff{md:ecDRO} to an equivalent linear conic optimization problem.
	This is an interesting question for future work.
	
	\medskip
        \noindent{\bf Acknowledgement.}
        Jiawang Nie is partially supported by the NSF grant DMS-2110780.

\end{document}